\documentclass[a4paper,12pt]{amsart}
\usepackage[utf8]{inputenc}
\usepackage{lmodern}
\usepackage{amsthm}
\usepackage{amssymb}
\usepackage{amsmath,amscd}
\usepackage{enumitem,mathrsfs}
\usepackage{color}
\usepackage{hyperref}
\usepackage{tikz-cd}

\newcommand{\lars}[2]{{\color{blue}
		\smallskip \centerline{ \fbox{\vbox{\noindent Lars, #1: #2 }}} \smallskip }}

\usepackage[top=3 cm, bottom=3 cm, left=2.7 cm, right=2.7cm]{geometry}

\title[Unlikely Intersections in Semiabelian Varieties]{Unlikely Intersections of Curves with Algebraic Subgroups in Semiabelian Varieties}

\newtheorem*{conjecture}{Conjecture}
\newtheorem{lem}{Lemma}[section]
\newtheorem{prop}{Proposition}[section]
\newtheorem{thm}{Theorem}[section]

\newcommand{\C}{\mathbb{C}}
\newcommand{\Q}{\mathbb{Q}}
\newcommand{\QQ}{\overline{\mathbb{Q}}}
\newcommand{\KK}{\mathbb{K}}
\newcommand{\Z}{\mathbb{Z}}

\newcommand{\R}{\mathbb{R}}

\newcommand{\cC}{C}

\newcommand{\Gal}{\mathrm{Gal}}
\newcommand{\codim}{\mathrm{codim}}
\newcommand{\pr}{\mathrm{pr}}
\usepackage{todonotes}
\newcommand{\commF}[2][]{\todo[#1,color=green]{F: #2}}

\newcommand{\ord}{\mathrm{ord}}

\DeclareMathOperator{\Lie}{\mathrm{Lie}}

\newcommand{\hhat}{\widehat{h}}

\newcommand{\Gm}{\mathbb{G}_m}

\subjclass[2010]{Primary 11G10; Secondary 03C64, 11G50, 14G40, 14K99}
\keywords{Unlikely intersections, Zilber-Pink Conjecture, semiabelian varieties, heights}

\author{Fabrizio Barroero}
\address{Dipartimento di Matematica e Fisica \\
	Università degli Studi Roma 3 \\
	L.go S. L. Murialdo 1 \\
	00146 Roma \\
	Italy}
\email{fbarroero@gmail.com}

\author{Lars K\"uhne}
\address{Institut for Matematiske Fag \\
	Universitetsparken 5 \\
	2100 København Ø \\
	Denmark\\
and 
Institut für Algebra, Zahlentheorie und Diskrete Mathematik \\ Leibnitz Universit\"at Hannover\\
Welfengarten 1\\
30167 Hannover\\
Germany}
\email{lk@math.ku.dk, 
	kuehne@math.uni-hannover.de}

\author{Harry Schmidt}
\address{Departement Mathematik und Informatik \\ 
	Universit\"at Basel\\
Spiegelgasse 1 \\
4051 Basel \\ 
Switzerland}
\email{harry.schmidt@unibas.ch}

\date{\today}

\begin{document}
	
\begin{abstract} 
	Let $G$ be a semiabelian variety and $\cC$ a curve in $G$ that is not contained in a proper algebraic subgroup of $G$. In this situation, conjectures of Pink and Zilber imply that there are at most finitely many points contained in the so-called unlikely intersections of $\cC$ with subgroups of codimension at least $2$. 
In this note, we establish this assertion for general semiabelian varieties over $\QQ$. This extends results of Maurin and Bombieri, Habegger, Masser, and Zannier in the toric case as well as Habegger and Pila in the abelian case.
\end{abstract}

\maketitle

\section{Introduction}

Let $G$ be a semiabelian variety defined over $\QQ$ and $X \subseteq G$ an algebraic subvariety. In this article, we study the intersections of $X$ with algebraic subgroups $H \subseteq G$ having codimension at least $\dim(X)+1$. For reasons of dimension, one does not expect that such a subgroup $H$ intersects $X$ at all. However, such ``unlikely intersections'' can appear and interesting phenomena arise when intersecting $X \subseteq G$ with the countable union $G^{[\dim(X)+1]}$ of {all} algebraic subgroups having codimension $\geq \dim(X)+1$. For example, it is not clear a priori whether the intersection $X \cap G^{[\dim(X)+1]}$ is Zariski-dense in $X$ or not. However, conjectures of Pink \cite{Pink2005a} and Zilber \cite{Zilber2002} imply the following.
\begin{conjecture}[Unlikely Intersection Conjecture, (UIC)]
If $X$ is not contained in a proper algebraic subgroup of $G$, then $X \cap G^{[\dim(X)+1]}$ is not Zariski-dense in $X$.
\end{conjecture}

One immediately realizes that, in the case $X$ is a hypersurface, this statement is exactly the Manin-Mumford conjecture which was proved by Laurent \cite{Laurent84}, Raynaud \cite{Raynaud83} and Hindry \cite{Hindry1988} in the multiplicative, abelian and general semiabelian case respectively.

Some special cases of the above conjecture have been already mentioned in a pioneering work of Bombieri, Masser, and Zannier \cite{Bombieri1999}, in which they proved that a curve $\cC$ in $G=\Gm^t$ has finite intersection with $G^{[2]}$ under the stronger assumption that $\cC$ is not contained in any translate of a proper algebraic subgroup of $G$. Maurin \cite{Maurin} gave a general proof with the weaker necessary assumption that $\cC$ is not contained in a proper algebraic subgroup, using the generalized Vojta inequality of Rémond \cite{Remond2005}. An alternative proof has been given by Bombieri, Habegger, Masser, and Zannier \cite{BHMZ}, relying on Habegger's proof of the Bounded Height Conjecture for algebraic tori \cite{Habegger2009}. 

After several partial results (\cite{RemVia}, \cite{Ratazzi08}, \cite{Carrizosa}, \cite{Viada2008}, \cite{galateau2010}) in this direction, Habegger and Pila \cite{HP} eventually proved (UIC) for curves in abelian varieties, using o-minimal counting techniques. 

The main purpose of this note is to establish the following generalization of the results in \cite{Maurin, HP}, which amounts to the full (UIC) in case $X$ is an algebraic curve.

\begin{thm}\label{mainthm}
	Let $G$ be a semiabelian variety and $\cC \subseteq G$ an irreducible curve not contained in a proper algebraic subgroup of $G$. Suppose $\cC$ and $G$ are defined over a number field $\KK$. Then $\cC \cap G^{[2]}$ is finite. 
\end{thm}

If $C$ is contained in a proper algebraic subgroup $H$ of $G$, the intersection of $C$ with $G^{[2]}$ can be infinite. This is easy to see in the toric case $G = \Gm^t$. Passing to the component of $H$ containing $C$, which can be assumed to be $ \approx \Gm^{t^\prime}$ possibly after raising $C$ to an appropriate power, this reduces to the well-known fact that the intersection of a curve $C^\prime \subseteq \Gm^{t^\prime}$ with $(\Gm^{t^\prime})^{[1]}$ is infinite. Indeed, any non-constant rational function $f$ on $C^\prime$ attains infinitely many roots of unity as values, the preimages of which are elements of $C^\prime \cap (\Gm^{t^\prime})^{[1]}$ if $f$ is the restriction of a standard coordinate function on $\Gm^{t^\prime}$. In conclusion, our result is essentially optimal. 

We also note that the above theorem implies the Mordell-Lang conjecture for curves in semiabelian varieties (\cite[Theorem 5.3]{Pink2005}), yielding hence a special case of a well-known result of McQuillan \cite{McQuillan1995}. However, we make use of the Mordell-Lang conjecture -- in the proof of Lemma \ref{lemfink} -- for the intersection of curves in abelian varieties with their Mordell-Weil group. This is also a source of non-effectivity in our main result. 

Our main Theorem and most of the above-mentioned results concern semiabelian varieties defined over the algebraic numbers due to the arithmetical methods employed in the proofs. Of course, one might look at curves and semiabelian varieties defined over larger fields. In \cite{BMZ08}, the authors show that Maurin's Theorem holds for complex curves, while the first author, in collaboration with Dill, extended the work of Habegger and Pila to curves and abelian varieties over the complex numbers in \cite{BD}.  Finally, in \cite{BD20}, they extended our Theorem \ref{mainthm} to complex curves lying in semiabelian varieties over $\QQ$.

Our proof of Theorem \ref{mainthm} follows the strategy employed by Bombieri, Habegger, Masser and Zannier in \cite{BHMZ}, relying on the Bounded Height Conjecture proved for semiabelian varieties by the second author \cite{Kuehne2020}. 
However, new difficulties arise in its implementation for semiabelian varieties instead of algebraic tori. 

First, semiabelian varieties lack Poincaré reducibility, which is true for tori as well as for abelian varieties. This makes it often necessary to avoid its usage by other tools, mostly through auxiliary quotient constructions.

Second, we have to prove that the points in $C \cap G^{[2]}$ satisfy a Northcott property (i.e., there are only finitely many such points of height $\leq B$ for any constant $B$); see Proposition \ref{lemfinbh}. In the toric case $G = \Gm^t$, this was proven in \cite[Lemma 1]{Bombieri2006}, but as this approach does not seem to generalize to semiabelian varieties, we adapt instead a counting argument from the proof of (UIC) for curves in abelian varieties given by Habegger and Pila \cite{HP}. The proof of this counting argument involves linear forms in logarithms on semiabelian varieties \cite{Gaudron2005, Wuestholz1989} as well as Pila and Wilkie's o-minimal counting Theorem \cite{PW} and its refinement \cite{HP} and Ax's Theorem \cite{Ax}.

Third, technical difficulties related to semiabelian varieties appear throughout the argument. A rather unexpected complication comes from the mixed structure of semiabelian varieties. In order to establish a close relation between the algebraic degree of a subgroup $H \subseteq G$ and the covolume of the associated period lattice $\Omega_H$ in its $\mathbb R$-linear span $V_H$ (Lemma \ref{lemma::covolume_degree}), it is necessary to choose a specific metric on the Lie algebra of $G(\mathbb C)$.

Our result has applications in the context of Pell equations over polynomial rings. In fact, Masser and Zannier \cite{Masser2015} had the astonishing insight that there is a connection between (UIC) and the (local) solvability of families of such Pell equations over a base curve. As an example, let us consider the two equations
\begin{align}\label{pell2}
A^2 -((T^3 +1)X^4+3T^2X^3 +3TX^2 + X)B^2 = X-1
\end{align}
and
\begin{align}\label{pell1}
A^2 -(X-1/(1-T))^2((T^3 +1)X^4 +3T^2X^3 +3TX^2 + X)B^2 = X-1
\end{align}
where we consider $T$ as the coordinate on the base curve ${\mathbb A}_{\mathbb C}^1\setminus \{1\}$ and seek solutions $(A,B) \in \C[X]^2$ after the specialization $T = t$, $t \in \C\setminus \{1\}$. Consider the family of curves
\begin{align*}
\mathcal{C}: \ Y^2 = (T^3 +1)X^4 +3T^2X^3 +3TX^2 + X
\end{align*} 
over $U= \mathbb{A}_\mathbb{C}^1\setminus \{(-1)^{1/3},1\}$ and, for each $t \in U(\mathbb{C})$, the points
$$
P^{\pm}_t = \left(1, \pm\sqrt{(t+1)^3+1}\right)
$$
of $\mathcal{C}_t$. Denote by $\infty^\pm_t$ the two points at infinity of a projective non-singular model $\overline{\mathcal{C}}_t$ of $\mathcal{C}_t$.
 Following the argument in \cite[Section 9]{BC20}, it can be seen that we have a solution $(A,B) \in \C[X]^2$ of (\ref{pell2}, $T = t$) if and only if   $P_t= [P^+_t-\infty^-_t] $ is a multiple of  $Q_t = [\infty^+_t-\infty^-_t]$ on the Jacobian $\mathrm{Jac}(\overline{\mathcal{C}}_t)$. 
 Looking at the degree of $T$ in \eqref{pell2}, it is clear that the equation is not identically solvable and thus $P_t$ is not a multiple of $Q_t$ identically. On the other hand, one can show using Siegel’s Theorem for integral points on curves over function fields (see the end of Section 10 of \cite{BC20} or p.~68 of \cite{Zannier}) that there are infinitely many $t \in \C$ for which $P_t$ is a multiple of $Q_t$. Thus the equation (\ref{pell2}, $T = t$) is solvable for infinitely many $t \in \C$. 

In contrast to this, our Theorem \ref{mainthm} implies that there are at most finitely many $t \in \mathbb C$ such that the other equation (\ref{pell1}, $T=t$) admits a solution $(A,B) \in \C[X]^2$. However note that the polynomial in front of $B^2$ in (\ref{pell1}) is not square-free. This forces one to consider not just abelian varieties but also linear extensions of those. See \cite{bertrandpell} or \cite{schmidtpell} for a short description and some examples involving multiplicative as well as additive extensions. 
 The generic fiber of the family $\mathcal{C}$ is birationally equivalent to the elliptic curve
\begin{equation*}
	E: \ Y^2 = X^3 + 1,
\end{equation*}
inducing an isomorphism between the Jacobians of their projective non-singular models. 

We write $\varphi_T: \mathcal{C}\dashrightarrow E$ for the birational map over $\overline{\Q}(T)$. 
It is explicitly given by sending $(X,Y)$ to $(1/X + T, Y/X^2)$ and 
we first compute that 
$$\varphi_T\left(  \frac1{1-T}, \pm\frac{\sqrt{2}}{(1-T)^2} \right) = (1, \pm \sqrt{2}).$$
Note that $$\varphi_T(P_T^{\pm}) = \left(1 + T,\pm \sqrt{ (1 + T)^3 + 1}\right) \ \ \text{and} \ \ \varphi_T(\infty_T^{\pm}) = \left(T, \pm \sqrt{T^3 + 1}\right).$$
Write $\mathrm{Jac}_\mathfrak{m}(E)$ for the generalized Jacobian of $E$ with respect to the modulus $$\mathfrak{m} = (1,\sqrt{2}) + (1,-\sqrt{2}).$$ From Lutz-Nagell for number fields \cite[VIII, Exercise 8.11]{Silverman}, it  follows that $(1, \pm\sqrt{2})$ is not torsion on $E$. Consulting the references cited above, we conclude that  $\mathrm{Jac}_\mathfrak{m}(E)$ is isomorphic to a fixed non-split extension $G$ of $E$ by $\mathbb{G}_m$. We write 
$$\overline{P}_t = [\varphi_t(P_t^+) -\varphi_t(\infty_t^-)]_{\mathfrak m},\ \overline{Q}_t = [\varphi_t(\infty^+_t)-\varphi_t(\infty^-_t)]_{\mathfrak m} \in \mathrm{Jac}_\mathfrak{m}(E).$$ 
Now suppose we have a solution of (\ref{pell1}, $T = t$). Then we get a rational function $f$ on $E$ whose divisor is a linear combination of the divisors $[\varphi_t(P_t^+) -\varphi_t(\infty_t^-)]$ and $[\varphi_t(\infty^+_t)-\varphi_t(\infty^-_t)]$. Moreover $f(1, \sqrt{2}) = f(1,-\sqrt{2}) \neq 0$. This implies that $\overline{P}_t, \overline{Q}_t$ satisfy a linear relation over $\mathbb{Z}$. As $t$ varies over the points of the base curve $U$, the point $(\overline{P}_t, \overline{Q}_t)$ defines an irreducible curve $C$ in $G^2$ and each linear relation defines a subgroup of $G^2$ of codimension 2. Clearly, both $G$ and $C$ are defined over the algebraic numbers. 
In addition, we can check that the point $\varphi_{-2}(P^{\pm}_{-2}) = (-1,0)$ is torsion on $E$ while, again using Lutz-Nagell, one can check that $\varphi_{-2}(\infty_{-2}^{\pm})$ is not torsion. Thus $\overline{P_t}, \overline{Q}_t \in \mathrm{Jac}_{\mathfrak m}(E)$ satisfy no constant linear relation  and $C$ is not contained in a subgroup of $G^2$ of positive codimension. This reduces the finiteness assertion to our Theorem \ref{mainthm}.

Another application of this kind of results has recently appeared in \cite{BCMOS} and concerns multiplicative dependence of values of rational functions and linear dependence of points of elliptic curves after reduction modulo primes. There are three main results in \cite{BCMOS} respectively for $\Gm^t$, $E^g$ and $\Gm^t\times E^g$, where $E$ is an elliptic curve over $\Q$, and they use the appropriate special case of the statement of Theorem \ref{mainthm}. While for the first two cases, the desired results were already in the literature (\cite{Maurin}  and \cite{Viada2008}, \cite{galateau2010}), for the third the authors prove a weaker case of Theorem \ref{mainthm} where the projections of the curve on the multiplicative and the elliptic factors cannot both be contained in proper coset. In this way, the bounded height results contained in \cite{Bombieri1999} and \cite{Viada03} are sufficient and the use of the strategy of \cite{BHMZ} can be avoided. With the result of this article we can remove that superfluous hypothesis in Theorem 2.8 of \cite{BCMOS} and in its corollaries.

\section{Preliminaries}\label{preliminaries}

Throughout this section, $G$ is a semiabelian variety defined over a field $k$. Recall that this means that $G$ is a connected smooth algebraic $k$-group that is the extension 
\begin{equation*}
\begin{tikzcd} 
0 \ar[r] & T \ar[r] & G \ar[r] & A \ar[r] & 0
\end{tikzcd}
\end{equation*}
of an abelian variety $A$ of dimension $g$ by an algebraic torus $T$. We always assume that $T=\Gm^t$; such an identification can always be made if $k=\mathbb C$ or after base change to a finite extension of $k$. 

In this section we collect and prove several basic results that we are going to need later.

\subsection{The open anomalous locus} 
A coset in $G$ is the translate $H + p$ of a connected algebraic subgroup $H$ by a closed point $p$ of $G$. We usually write cosets of $G$ in the form $p+H$. Let $V$ be an irreducible subvariety of $G$. A positive-dimensional irreducible subvariety $W\subseteq V$ is called $G$-\emph{anomalous} (or simply \emph{anomalous}) in $V$ if it is contained in a coset $p+H$ with
\begin{equation*}
\dim V - \dim W  < \dim G -\dim H .
\end{equation*}
We let $V^{\mathrm{oa},G}$ (or simply $V^{\mathrm{oa}}$) to be the complement in $V$ of the union of all anomalous subvarieties of $V$.

The following theorem is a generalization of Theorem 1.4 of \cite{BMZ07}.

\begin{thm} [Structure Theorem]\label{ST}
	\label{structure}
	Let $V\subseteq G$ be an irreducible subvariety of positive dimension.
	\begin{enumerate}[label=(\alph*)]
	\item For any proper semiabelian subvariety $H$ of $G$, 
		the union $\mathscr{Z}_H$ of all subvarieties W of V contained in any coset of $H$ with
		\begin{equation}\label{st2}
		\dim W >  \max \{0, \dim V +\dim H - \dim G \}
		\end{equation} 
		is a closed subset of $V$, and the set $H + \mathscr{Z}_H$ is not dense in $G$.
		\item There is a finite collection $\varPhi_V$ of such proper semiabelian subvarieties $H$ such that every maximal anomalous subvariety $W$ of $V$ is a component of $V \cap (p+H)$ for some $H$ in $\varPhi_V$ satisfying 
		 \eqref{st2} and some $p$ in $\mathscr{Z}_H$; and $V^{oa}$ is obtained from $V$ by removing the $\mathscr{Z}_H$ for all $H$ in $\varPhi_V$. In particular $V^{oa}$ is open in $V$.
	\end{enumerate}
\end{thm}

\begin{proof}
	This is essentially contained in the proof of \cite[Corollary 2.4]{C-L} and follows from work of Kirby \cite{Kirby} and the Fiber Dimension Theorem (see, e.g., \cite[p.~8]{BMZ07}).
	
	Let us first prove part (a). Consider the projection $\psi:G\rightarrow G/H$ and its restriction $\psi_V:V\rightarrow \psi (V)$ to $V$. Then, $V \cap (p+H)= \psi_V^{-1}(\psi_V(p))$ for all $p \in V$. Now, the set $\mathscr{Z}_H$ consists of the fibers of $\psi_V$ of dimension
	$$> \max \{0,  \dim V -\dim(G/H) \}. $$
	 and is therefore closed in $V$ by the Fiber Dimension Theorem. The set $H + \mathscr{Z}_H$ is nothing but the union of all cosets of $H$ that give rise to an anomalous component in $V$. We have $\psi(H + \mathscr{Z}_H)= \psi(\mathscr{Z}_H)\subseteq \psi(V)$. If $H + \mathscr{Z}_H$ were dense in $G$, then $\psi( \mathscr{Z}_H)$ would be dense in $G/H$. This is only possible if $\psi(V)$ is dense in $G/H$. This implies that $$\dim (G/H)=\dim G-\dim H= \dim \psi(V) \leq \dim V  .$$ But now the Fiber Dimension Theorem gives an open dense $U \subseteq \psi(V)$ whose points $q$ satisfy $$\dim(\psi^{-1}(q))=\dim V -\dim \psi(V)=\dim V -\dim (G/H)\geq 0.$$ The set $U$ must then be disjoint from $\psi(\mathscr{Z}_H)$ and therefore the latter cannot be dense in $G/H$.
	
	We now turn to part (b). We recall \cite[Theorem 2.3]{C-L}, which is a corollary of the theorem on \cite[p.~449]{Kirby}. If V is an irreducible subvariety of $G$ then there exists a finite family $\varPhi_V$ of proper semiabelian subvarieties such that for all cosets $p+K$ of a semiabelian subvariety $K$ and all anomalous components $W$ of $V\cap (p+K)$ there exists an $H\in \varPhi_V$ and $q\in G$ so that $W\subseteq q+H$ and
	\begin{equation}\label{eqdim}
	\dim H+\dim W =\dim K +\dim W'.
	\end{equation}
	where $W'$ is the irreducible component of $V\cap(q+H)$ containing $W$.
	Now, let $W$ be a maximal anomalous component arising from an intersection $V\cap (p+K)$. We can suppose that $K$ is the smallest semiabelian subvariety so that $W\subseteq p+K$. We will show that $K \in \varPhi_V$, where $\varPhi_V$ is as in the above statement. The statement actually gives an $H\in \varPhi_V$ with \eqref{eqdim}. We show that $K=H$. By assumption we have $K\subseteq H$. Let $W'$ be the irreducible component of $V\cap (q+H)$ containing $W$ and suppose $H\neq K$ so $W\subsetneq W'$. By the maximality of $W$, $W'$ cannot be anomalous and we have 
	$$
	\dim W' \leq \dim V + \dim H- \dim G.
	$$
	So, by \eqref{eqdim}, we have
	$$
	\dim W\leq \dim V +\dim K -\dim G,
	$$
	which contradicts the fact that $W$ was an atypical component of the intersection $V\cap (p+K)$.
\end{proof}

\subsection{Degrees of subgroups and periods}
\label{subsection::degrees}

In this subsection, we consider a semiabelian variety $G$ defined over $\mathbb{C}$. We identify smooth subvarieties $X \subseteq G$ with the complex analytic manifolds associated to them through analytification \cite{Serre1955--1956}. 

For the degree computations in this subsection, it is convenient to work with an explicit compactification $\overline{G}_0$ of $G$ and a specific ample line bundle $L_0$ on $\overline{G}_0$. Nevertheless, the results obtained are independent of this specific choice. Let us consider the compactification $\overline{G}_0$ of $G$, the maps $\overline{[n]}: \overline{G}_0 \rightarrow \overline{G}_0$ and $\overline{\pi}: \overline{G}_0 \rightarrow A$, and the line bundle $M=M_{\overline{G}_0}$ from \cite[Construction 5]{Kuehne2020}. Furthermore, we choose a very ample symmetric line bundle $N$ on the abelian quotient $A$ of $G$. We then choose the line bundle $L_0 = M \otimes \overline{\pi}^{\ast} N$ on $\overline{G}_0$. Indeed, this is an ample line bundle by \cite[Lemma 3]{Kuehne2020}. As $\overline{[2]}^\ast M = M^{\otimes 2}$ and $\overline{[2]}^\ast N = N^{\otimes 4}$, we can endow both $M$ and $N$ with canonical hermitian metrics $h_M$ and $h_N$ by using \cite[Theorem 2.2]{Zhang1995a}. These metrics are unique up to a non-zero scalar.

Let $\exp_{G}: \mathrm{Lie}(G) \rightarrow G$ be the (complex) Lie group exponential of $G$ and write $\Omega_G = \exp_G^{-1}(\{0_G\})$ for the periods of $G$. It is well known that $\Omega_G$ is a discrete subgroup of $\mathrm{Lie}(G)$ of rank $2g+t$ and we write $V_G \subseteq \mathrm{Lie}(G)$ for its $\mathbb{R}$-linear span, which coincides with the preimage of the maximal compact subgroup $K_G \subseteq G$ under $\exp_G$; we use analogous notations for other semiabelian varieties. 

The exact sequence
\begin{equation*}
\begin{tikzcd} 
0 \ar[r] & \Gm^t \ar[r] & G \ar[r] & A \ar[r] & 0
\end{tikzcd}
\end{equation*}
induces an exact sequence
\begin{equation*}
\begin{tikzcd} 
0 \ar[r] & \mathrm{Lie}(\Gm^t) \ar[r] &  \mathrm{Lie}(G) \ar[r] & \mathrm{Lie}(A) \ar[r] & 0
\end{tikzcd}
\end{equation*}
restricting to a sequence
\begin{equation*}
\begin{tikzcd} 
0 \ar[r] & \Omega_{\Gm^t} \ar[r] & \Omega_G \ar[r] & \Omega_A \ar[r] & 0
\end{tikzcd}
\end{equation*}
between the respective period lattices. Finally, we obtain an exact sequence
\begin{equation}
\label{equation::seq_tosplit}
\begin{tikzcd} 
0 \ar[r] &  V_{\Gm^t} \ar[r] & V_G \ar[r] & V_A \ar[r] & 0
\end{tikzcd}
\end{equation}
of $\mathbb{R}$-vector spaces.

Starting from the $(1,1)$-form $c_1(N,h_N)$, we can define a symmetric, positive definite, $\mathbb{R}$-bilinear form on $V_A$; in fact, the pullback $\exp_A^{\ast}c_1(N,h_N)$ is an invariant positive definite $(1,1)$-form on the $\mathbb{C}$-vector space $\mathrm{Lie}(A)=V_A$. We can hence use the one-to-one correspondence between Hermitian forms and symmetric $\mathbb{R}$-bilinear forms (see e.g.\ \cite[Section 4]{Kuehne2020}). Write $g_{\mathrm{ab}}$ for the invariant Riemannian metric obtained on $V_A$ in this way. We also describe an invariant Riemannian metric $g_{\mathrm{tor}}$ on $V_{\Gm^t}$ as follows: The standard product decomposition $\Gm^t = \Gm \times \cdots \times \Gm$ gives rise to a product decomposition 
\begin{equation*}
\Omega_{\Gm^t} = (2\pi i)\Z \times \cdots \times (2\pi i)\Z \subset \mathrm{Lie}(\Gm)^t = \mathrm{Lie}(\Gm^t)
\end{equation*}
of the period lattice. We let $g_{\mathrm{tor}}$ be the unique invariant Riemannian metric on $V_{\Gm^t}$ such that
\begin{equation*}
\left(2\pi i,0,\dots,0 \right), \left(0,2\pi i,\dots,0 \right), \dots, \left(0,0,\dots, 2\pi i \right), 
\end{equation*}
form a $g_{\mathrm{tor}}$-orthonormal basis of $V_{\Gm^t}$.

In order to obtain a metric $g_G$ on $V_G$ from $g_{\mathrm{ab}}$ and $g_{\mathrm{tor}}$, we next describe a canonical splitting $\sigma: V_G \rightarrow V_{\mathbb{G}_m^t}$ of \eqref{equation::seq_tosplit}. Consider the additive homomorphisms $\lambda_1,\dots,\lambda_t: G(\mathbb{C}) \rightarrow \mathbb{R}$ afforded by \cite[Lemma 14]{Kuehne2020}. We recall also that
\begin{equation*}
\{ x \in G(\mathbb{C}) : \lambda_1(x)=\lambda_2(x)= \cdots = \lambda_{t}(x) = 0 \} 
\end{equation*}
coincides with the maximal compact subgroup $K_G$ of $G(\mathbb C)$. In addition, the functions $\lambda_1,\dots,\lambda_t$ satisfy a certain functoriality: Let $\varphi: H \rightarrow G$ be a homomorphism of semiabelian varieties inducing a homomorphism $\varphi_{\mathrm{tor}}: \Gm^{t^\prime} \rightarrow \Gm^t$ of their maximal tori. Write $X_i$ (resp.\ $Y_j$) for the standard algebraic coordinates on $\Gm^{t}$ (resp.\ $\Gm^{t^\prime}$) so that we have $\varphi_{\mathrm{tor}}^\ast(X_u)=Y_1^{a_{ u1}}\cdots Y_{t^\prime}^{a_{u t^\prime }}$ with integers $a_{uv}$ ($1 \leq u \leq t$, $1 \leq v \leq t^\prime$). Let $\lambda^\prime_1,\dots,\lambda^\prime_{t^\prime}: H(\mathbb{C}) \rightarrow \mathbb{R}$ be given by invoking \cite[Lemma 14]{Kuehne2020} for $H$. On $H(\mathbb{C})$ there is then an identity
\begin{equation} \label{equation::lambda_functoriality}
\lambda_u \circ \varphi =  a_{ u1} \lambda_1^\prime + a_{ u2} \lambda_2^\prime  + \cdots + a_{ut^\prime } \lambda^\prime_{t^\prime};
\end{equation}
all these facts can be found with proofs in \cite[Section 5.1]{Kuehne2020}.

The compositions $\lambda_j \circ \exp_G: \mathrm{Lie}(G) \rightarrow \mathbb{R}$, $j=1,\dots,t$, are $\mathbb{R}$-linear and their common zero locus is precisely $\mathrm{Lie}(K_G)$. As $G$ is a complex Lie group, multiplication-by-$i$ ($i = \sqrt{-1})$ induces an $\mathbb{R}$-linear map $I:\mathrm{Lie}(G) \rightarrow \mathrm{Lie}(G)$. (Note that $I$ does not preserve $\mathrm{Lie}(K_G)$.) For each $u \in \{1,\dots,t\}$, we set
\begin{equation*}
	\kappa_u = (-i \lambda_u)\circ \exp_G \circ \, I \! : V_G \longrightarrow (2\pi i)\mathbb{R}.
\end{equation*}
\begin{lem}
\label{lemma::splitting}
Identifying $(2\pi i \mathbb{R})^t$ with $V_{\Gm^t}$ in the obvious way, the $\mathbb{R}$-linear map
\begin{equation}
\label{equation::splitting}
\sigma = \kappa_1 \times \cdots \times \kappa_t: V_G \longrightarrow ((2\pi i)\mathbb{R})^t = V_{\Gm^t}
\end{equation}
splits the exact sequence \eqref{equation::seq_tosplit} on the left. 

This splitting is compatible with passing to subgroups: If $H \subseteq G$ is an algebraic subgroup with maximal torus $T \subseteq \Gm^t$, then we have $\sigma(V_H) \subseteq V_T$ in $V_{\Gm^t}$.
\end{lem}

\begin{proof}
Writing $\iota: V_{\Gm^t} \hookrightarrow V_{G}$ for the inclusion from \eqref{equation::seq_tosplit}, we have to show that $\sigma \circ \iota$ is the identity on $V_{\Gm^t}$. By functoriality \eqref{equation::lambda_functoriality}, we can reduce to the case $G=\Gm^t$ for this. We can furthermore restrict to the case $t=1$ because of the product structure. In this case, we have $\lambda_1(z) = \log |z|$ where $z$ is the standard complex coordinate on $\Gm$ (compare again \cite[Section 5.1]{Kuehne2020}). In addition, $\exp_G$ is just the ordinary exponential function. For each real number $r$, we have thus
\begin{equation*}
	\kappa_1 (2\pi i \cdot r) = -i \log |\exp(-2\pi r)|= 2 \pi i  \cdot r,
\end{equation*}
which completes the proof of the first assertion.

The second assertion follows from functoriality \eqref{equation::lambda_functoriality}: Let $X_1,\dots,X_t$ be the standard algebraic coordinates on $\Gm^t$. If $X_1^{a_1}\cdots X_t^{a_t}-1$ vanishes on the subtorus $T \subseteq \Gm^t$, \eqref{equation::lambda_functoriality} guarantees that $a_1\lambda_1 + \dots + a_t\lambda_t$ vanishes on $H(\mathbb{C})$. This yields a corresponding linear relation on the image of $\sigma|_H$. Varying the binomial $X_1^{a_1} \cdots X_t^{a_t}-1$ vanishing on $T$, the subspace $V_T$ is precisely cut out by these linear relations. Thus we deduce $\sigma(V_H) \subseteq V_T$.
\end{proof}

The splitting \eqref{equation::splitting} induces an $\mathbb{R}$-linear isomorphism $V_G = V_{\Gm^t} \times V_A$. Using this isomorphism, we obtain a Riemannian metric $g_G = g_{\mathrm{tor}} \times g_{\mathrm{ab}}$ on $V_G$. This allows us to consider the covolume $\mathrm{covol}_{g_G}(\Omega \subset V)$ of a discrete subgroup $\Omega$ of $V_G$ inside its $\R$-span $V$.

In what follows, for a semiabelian subvariety $H$ of $G$, $\overline{H}$ indicates the closure of $H$ in $\overline{G}$. Furthermore, we fix an arbitrary ample line bundle $L$ on $\overline{G}$ once and for all.

\begin{lem} 
\label{lemma::covolume_degree}	
Let $H \subseteq G$ be a semiabelian subvariety with maximal torus $T \subseteq \Gm^t$ and maximal abelian quotient $B \subseteq A$. Then,
\begin{equation*}
c_1 \deg_L(\overline{H})\leq \mathrm{covol}_{g_G}(\Omega_H \subset V_H) \leq c_2 \deg_L(\overline{H})
\end{equation*}
for constants $c_1,c_2 >0$ that depend only on $\dim(G)$ and $L$.
\end{lem}

If $G$ is an abelian variety, the above inequalities can be sharpened to an equality \cite[Lemma 3.1]{HP}. In the toric case, this is unfortunately not possible. In fact, the one-dimensional subtorus $T \subseteq \Gm^2$ determined by the equation $X_1^{a_1}X_2^{a_2}=1$, $\gcd(a_1,a_2)=1$, in standard coordinates $X_1,X_2$ on $\Gm^2$ has degree $|a_1|+|a_2|$ with respect to the line bundle $\pr_1^\ast \mathcal{O}(1) \otimes \pr_2^\ast \mathcal{O}(1)$ on the standard compactification $\Gm^2 \hookrightarrow \mathbb{P}^1 \times \mathbb{P}^1$, whereas $\mathrm{covol}(\Omega_T \subset V_T)=(|a_1|^2+|a_2|^2)^{1/2}$.

\begin{proof}[Proof of Lemma \ref{lemma::covolume_degree}:] Without loss of generality, we can assume that $\overline{G}=\overline{G}_0$ and $L=L_0$ for the proof of the lemma.
	
The corresponding diagram of algebraic groups induces a commutative diagram with exact rows
\begin{equation*}
\begin{tikzcd}
0 \ar[r] & V_T \ar[r] \ar[d, hook] & V_H \ar[r] \ar[d, hook] & V_B \ar[r] \ar[d,hook] & 0 \\
0 \ar[r] & V_{\Gm^t} \ar[r] & V_G \ar[r] & V_A \ar[r] & 0
\end{tikzcd}
\end{equation*}
consisting of $\mathbb{R}$-linear vector spaces. By Lemma \ref{lemma::splitting}, the restriction $\sigma|_{V_H}: V_H \rightarrow V_T$ is a splitting of the upper row. It hence induces a $g_G$-orthogonal decomposition $V_H = V_{T} \times V_B$ so that
\begin{equation*}
\mathrm{covol}_{g_G}(\Omega_H \subset V_H) = \mathrm{covol}_{g_{\mathrm{tor}}}(\Omega_T \subset V_T) \cdot \mathrm{covol}_{g_\mathrm{ab}}(\Omega_B \subset V_B).
\end{equation*}
By Lemma \ref{lemma::degreeprod} below, we have
\begin{equation*}
	\deg_L(\overline{H}) = \binom{\dim(H)}{\dim(T)} \deg_M(\overline{T}) \deg_N(B).
\end{equation*}
Therefore, it suffices to prove
\begin{equation}
\label{equation::covol_toric}
c_1 \deg_M(\overline{T}) \leq \mathrm{covol}_{g_\mathrm{tor}}(\Omega_T \subset V_T) \leq c_2 \deg_M(\overline{T})
\end{equation}
and
\begin{equation}
\label{equation::covol_abelian}
(\dim B)!\mathrm{covol}_{g_\mathrm{ab}}(\Omega_B \subset V_B) = \deg_N(B).
\end{equation}

We first prove the inequalities in \eqref{equation::covol_toric}. Choose an isomorphism $T \approx \Gm^{t^\prime}$ so that the inclusion $\iota: T \hookrightarrow \Gm^t$ is described by
\begin{equation*}
\iota^\ast(X_u) = Y_1^{a_{u1}} \cdots Y_{t^\prime}^{a_{ut^\prime }},  u \in \{1,\dots,t\}, 
\end{equation*}
in standard coordinates $X_1,\dots,X_t$ (resp.\ $Y_1,\dots,Y_{t^\prime}$) on $\Gm^t$ (resp.\ $\Gm^{t^\prime}$). The period lattice $\Omega_T \subset \Omega_{\Gm^t}$ has the basis
\begin{equation*}
(2\pi i) \begin{pmatrix} a_{11} \\ \vdots \\ a_{t1} \end{pmatrix}, \ \cdots,
(2\pi i) \begin{pmatrix} a_{1t^\prime } \\ \vdots \\ a_{tt^\prime } \end{pmatrix},
\end{equation*}
and hence $\mathrm{covol}_{g_\mathrm{tor}}(\Omega_T \subset V_T)=|\det(\mathcal{A}^\top \mathcal{A})|^{1/2}$ where $\mathcal{A}=(a_{uv}) \in \mathbb{Z}^{t \times t^\prime}$ and $\mathcal{A}^\top$ is its transpose. Alternatively, the torus $T \subseteq \Gm^t$ is cut out by $(t-t^\prime)$ equations
\begin{equation}
\label{equation::torus_singleequation}
X_1^{b_{1v}} X_2^{b_{2v}} \cdots X_t^{b_{tv}} = 1, v \in \{1,\dots, t-t^\prime \},  
\end{equation}
Writing $\mathcal{B}$ for the matrix $(b_{uv}) \in \mathbb{Z}^{t \times (t-t^\prime)}$, the columns of $ \mathcal{B}$ are a $\mathbb{Z}$-basis of the sublattice in $\mathbb{Z}^n$ that is spanned by all elements orthogonal to the columns of $\mathcal{A}$. We hence have $|\det(\mathcal{B}^\top \mathcal{B})|=|\det(\mathcal{A}^\top \mathcal{A})|$ by \cite[Theorem 1.9.10]{Martinet2003}. By the Cauchy-Binet formula, we have
\begin{equation*}
|\det(\mathcal{B}^\top \mathcal{B})|\leq \binom{t}{t - t^\prime} \max_{\substack{\underline{u} \subseteq \{1, \dots, t\} \\ |\underline{u}| = t - t^\prime}} \left\{|\det (\mathcal{B}_{\underline{u}})|^2 \right\}
\end{equation*}
where $\mathcal{B}_{\underline{u}} \in \mathbb{Z}^{(t-t^\prime)\times(t-t^\prime)}$ is the $\underline{u}$-minor of $\mathcal{B}$. We claim that each $|\det(\mathcal{B}_{\underline{u}})|$ is less than $\deg_M(\overline{T})$. After renaming, we may and do assume that $\underline{u}=\{1,2,\dots, t-t^\prime \}$ for this purpose. Let $\mathcal{C}_{\underline{u}} = (c_{uv}) \in \mathbb{Z}^{(t-t^\prime) \times (t-t^\prime)}$ be an upper triangular matrix arising from $\mathcal{B}_{\underline{u}}$ by successive elementary row transformations. Then $\det(\mathcal{B}_{\underline{u}})=\det(\mathcal{C}_{\underline{u}})=c_{1,1}c_{2,2}\cdots c_{t-t^\prime,t-t^\prime}$. This is the same as the number of simple points one obtains by intersecting $\overline{T} \subseteq (\mathbb{P}^{1})^t$ with the linear hyperplanes
\begin{equation*}
X_{t-t^\prime+1} = \delta_1, \dots, X_{t} = \delta_{t^\prime},
\end{equation*}
for sufficiently generic $\delta_1,\dots, \delta_{t^\prime}$. This means nothing else than
\begin{equation}
\label{equation::intersection_number}
c_{1,1}c_{2,2}\cdots c_{t-t^\prime,t-t^\prime} = c_1(\mathrm{pr}_{t-t^\prime+1}^\ast \mathcal{O}_{\mathbb P^1}(1)) \cap \cdots \cap c_1(\mathrm{pr}_{t}^\ast \mathcal{O}_{\mathbb P^1}(1)) \cap [\overline{T}].
\end{equation}
By nefness, this is clearly bounded by
\begin{equation*}
\deg_M(\overline{T}) = 2^{t^\prime}(c_1(\mathrm{pr}_{1}^\ast \mathcal{O}_{\mathbb P^1}(1))+\cdots + c_1(\mathrm{pr}_{t}^\ast \mathcal{O}_{\mathbb P^1}(1)))^{t^\prime}\cap [\overline{T}].
\end{equation*}
This concludes the proof of the right inequality in \eqref{equation::covol_toric}.

For the left inequality in \eqref{equation::covol_toric}, we know from above that we may represent $T$ as being cut out by equations \eqref{equation::torus_singleequation} with $|\det(\mathcal{B}^\top \mathcal{B})|^{1/2} = \mathrm{covol}(\Omega_T \subset V_T)$. Furthermore, the Cauchy-Binet formula shows as well that
\begin{equation*}
|\det(\mathcal{B}^\top \mathcal{B})| \geq |\det(\mathcal{B}_{\underline{u}})|^2,
\end{equation*}
which bounds the intersection number in \eqref{equation::intersection_number}. The degree is just the sum of such numbers, whence the other inequality in \eqref{equation::covol_toric}.

For the equality \eqref{equation::covol_abelian}, we note that
\begin{equation*}
\mathrm{covol}_{g_\mathrm{ab}}(\Omega_B \subset V_B) = \int_{B(\mathbb{C})} \mathrm{vol}(g_{\mathrm{ab}}) = \frac{1}{(\dim B)!}\int_{B(\mathbb{C})} c_1(N,h_N)^{\wedge \dim B} = \frac{\deg_N(B)}{(\dim B)!};
\end{equation*}
the second equality is \cite[Lemma 3.8]{Voisin2007} and the third equality follows from the compatibility between algebraic and topological Chern classes acting on singular homology \cite[Proposition 19.1.2]{Fulton1998} and the fact that the topological Chern class of a hermitian line bundle is given by its Chern form (see e.g.\ \cite[Proposition on p.\ 141]{Griffiths1994}).
\end{proof}

The following is a straightforward application of basic intersection theory. In its proof, we use the notations from \cite[Chapters 1 and 2]{Fulton1998}. We also use the specific compactification $\overline{G}_0$ and the line bundle $L_0 =  M \otimes \overline{\pi}^{\ast} N$ from above. This is admissible because the lemma is only invoked in the proof of Lemma \ref{lemma::covolume_degree}.

\begin{lem}\label{lemma::degreeprod} Let $H \subseteq G$ be a semiabelian subvariety with maximal torus $T \subseteq \Gm^t$ and maximal abelian quotient $B \subseteq A$. Then,
	\begin{equation*}
		\deg_{L_0}(\overline{H}) = \binom{\dim(H)}{\dim(T)} \deg_M(\overline{T}) \deg_N(B).
	\end{equation*}
\end{lem}

\begin{proof}
	A straightforward computation yields that
	\begin{align*}
		\deg_L(\overline{H}) 
		&= \deg(c_1(L)^{\dim(H)} \cap [\overline{H}]) \\
		&= \deg(c_1(M \otimes \overline{\pi}^\ast N)^{\dim(H)} \cap [\overline{H}]) \\
		&= \sum_{i=0}^{\dim(H)} \binom{\dim(H)}{i} \deg ( c_1(M)^i \cap c_1(\overline{\pi}^\ast N)^{\dim(H)-i} \cap [\overline{H}]).
	\end{align*}
	We show next that all addends in this sum are zero except the one for $i=\dim(T)$. As the restriction $\overline{[2]}|_{\overline{H}}: \overline{H} \rightarrow \overline{H}$ has degree $2^{(\dim(T)+2\dim(B))}$, we have $\overline{[2]}_\ast [\overline{H}] = 2^{\dim(T)+2\dim(B)} [\overline{H}]$. An iterated application of the projection formula (\cite[Proposition 2.5 (c)]{Fulton1998}) to $\overline{[2]}|_{\overline{H}}$ and the line bundles $M$ and $\overline{\pi}^\ast N$ yields hence
	\begin{multline*}
		\overline{[2]}_\ast \left(c_1(\overline{[2]}^\ast M)^i \cap c_1(\overline{[2]}^\ast\overline{\pi}^\ast N)^{\dim(H)-i} \cap [\overline{H}]\right) 
		\\
		= 2^{\dim(T)+2\dim(B)} \left( c_1(M)^{i} \cap c_1(\overline{\pi}^\ast N)^{\dim(H)-i} \cap [\overline{H}] \right).
	\end{multline*}
	Taking the degree of these $0$-cycles, we obtain
	\begin{multline*}
		2^{i+2(\dim(H)-i)} \deg (c_1(M)^i \cap c_1(\overline{\pi}^\ast N)^{\dim(H)-i} \cap [\overline{H}]) \\
		=  2^{\dim(T)+2\dim(B)} \deg( c_1(M)^{i} \cap c_1(\overline{\pi}^\ast N)^{\dim(H)-i} \cap [\overline{H}]),
	\end{multline*}
	making also use of the homogeneities $\overline{[2]}^\ast M \approx M^{\otimes 2}$ and $\overline{[2]}^\ast (\overline{\pi}^\ast N) \approx (\overline{\pi}^\ast N)^{\otimes 4}$. This implies that
	\begin{equation*}
	\deg (c_1(M)^i \cap c_1(\overline{\pi}^\ast N)^{\dim(H)-i} \cap [\overline{H}]) = 0
	\end{equation*}
	unless $i=\dim(T)$. We deduce that 
	\begin{equation*}
	\deg_L(\overline{H}) = \binom{\dim(H)}{\dim(T)} \deg ( c_1(M)^{\dim(T)} \cap c_1(\overline{\pi}^\ast N)^{\dim(B)} \cap [\overline{H}]).
	\end{equation*}
	The restriction $\overline{\pi}|_{\overline{H}}: \overline{H} \rightarrow B$ is flat of relative dimension $\dim(T)$. We can therefore pull back cycle classes on $B$ to cycle classes on $\overline{H}$. In particular, we have $(\overline{\pi}|_{\overline{H}})^\ast ([B]) = [\overline{H}]$ and $(\overline{\pi}|_{\overline{H}})^\ast  ([p])=[\overline{\pi}^{-1}(p) \cap \overline{H}]$
	for every point $p \in B$. Since $N$ is ample, we can choose points $p_1,\dots,p_{m} \in A$, $m =  \deg_N(B)$, such that 
	\begin{equation*}
	c_1(N)^{\dim(B)} \cap [B] = [p_1] + [p_2] + \cdots + [p_m] \in A_0(B).
	\end{equation*}
	Using \cite[Proposition 2.5 (d)]{Fulton1998}, we obtain
	\begin{align}
	\label{equation::abelian_intersection}
	c_1({\overline{\pi}}^\ast N)^{\dim(B)} \cap [\overline{H}]
	=
	(\overline{\pi}|_{\overline{H}})^\ast (c_1(N)^{\dim(B)} \cap [B]) 
	=
	\sum_{i=1}^{m} [\overline{H} \cap \overline{\pi}^{-1}(p_i)].
	\end{align}
	By construction, there exists a non-canonical isomorphism between each fiber $(\overline{\pi}_{\overline{H}})^{-1}(p)$ and $\overline{T}$ such that the restriction of $M$ to $(\overline{\pi}_{\overline{H}})^{-1}(p)=\overline{H} \cap \overline{\pi}^{-1}(p_i)$ corresponds to the restriction of $M$ to $\overline{T}$. We infer that
	\begin{align*}
	\deg(c_1(M)^{\dim(T)} \cap [\overline{H} \cap \overline{\pi}^{-1}(p_i)])
	= \deg(c_1(M)^{\dim(T)} \cap [\overline{T}])
	= \deg_M(\overline{T}).
	\end{align*}
	Combining this with \eqref{equation::abelian_intersection}, we deduce
	\begin{equation*}
	\deg ( c_1(M)^{\dim(T)} \cap c_1(\overline{\pi}^\ast N)^{\dim(B)} \cap [\overline{H}]) = (m-n) \deg_M(\overline{T})
	= \deg_N(B) \deg_M(\overline{T}),
	\end{equation*}
	whence the assertion.
\end{proof}

In preparation for the next lemma, we need to fix a norm $\Vert \cdot \Vert: \mathrm{Lie}(G) \rightarrow \mathbb{R}^{\geq 0}$ such that $\Vert v \Vert = g_G(v,v)^{1/2}$ for all $v \in V_G$ (i.e., we fix a norm $\Vert \cdot \Vert$ on $\mathrm{Lie}(G)$ extending the norm induced by the Riemannian metric $g_G$ on $V_G$). For this purpose, let us note that the canonical decomposition $V_G = V_{\Gm^t} \times V_A$ introduced above extends to a canonical decomposition
\begin{equation}
\label{equation::lie_decomposition}
\mathrm{Lie}(G) = I V_{\Gm^t} \times  V_{\Gm^t} \times V_A = I V_{\Gm^t} \times V_{G}
\end{equation}
that is compatible with passing to subgroups $H \subseteq G$ by Lemma \ref{lemma::splitting}. For the sequel, we keep fixed an arbitrary norm $\Vert \cdot \Vert$ such that the decomposition $\mathrm{Lie}(G) = IV_{\Gm^t} \times V_G$ is orthogonal with respect to the associated bilinear form.

The following is an analog of Lemma 3.2 of \cite{HP}.

\begin{lem}
\label{lemma::minkowski}
	There exist constants $c_3,c_4$, depending only on $G$ and $L$, such that the following two assertions are true: 
\begin{enumerate}
	\item For each semiabelian subvariety $H \subseteq G$, the period lattice $\Omega_H \subseteq V_H$ has a basis $\omega_1,\dots,\omega_{2g^\prime+t^\prime}$ with $\Vert \omega_i \Vert \leq c_3 \deg_L(\overline{H})$ ($i \in \{1,\dots, 2g^\prime+t^\prime\}$).
	\item For each $v \in \Omega_G + \mathrm{Lie}(H)$, there exists a period $\omega \in \Omega_G$ such that $v-\omega \in \mathrm{Lie}(H)$ and $\Vert \omega \Vert \leq \Vert v \Vert + c_4 \deg_L(\overline{H})$.
\end{enumerate}
\end{lem}

\begin{proof}
	Again, we can assume that $\overline{G}=\overline{G}_0$ and $L=L_0$ without loss of generality.
	
	Given Lemma \ref{lemma::covolume_degree} above, the first part is a simple application of Minkowski's Theorem (see \cite[Lemma 3.2 (i)]{HP} for details).

	For the second part, consider the orthogonal projection $\psi: \mathrm{Lie}(G) = IV_{\Gm^t} \times V_G \rightarrow V_{G}$ and note that $\psi(\mathrm{Lie}(H)) = V_H$ for any semiabelian subvariety $H \subseteq G$.	Let now $\omega_0 \in \Omega_G$ such that $v - \omega_0 \in \mathrm{Lie}(H)$ and let further $\omega_1,\dots,\omega_{2g^\prime+t^\prime} $ be a basis of $ \Omega_H\subseteq V_H$ as in (1). There exist $r_1,\dots,r_{2g^\prime+t^\prime} \in \mathbb{R}$ such that $$\psi(v - \omega_0)=r_1 \omega_1 + \cdots + r_{2g^\prime+t^\prime} \omega_{2g^\prime+t^\prime}.$$ For each $i \in \{ 1, \dots, 2g^\prime+t^\prime\}$, let $n_i$ be the unique integer such that $0 \leq r_i-n_i < 1$. Setting $\omega = \omega_0 + n_1 \omega_1 + \cdots + n_{2g^\prime+t^\prime}\omega_{2g^\prime+t^\prime} \in V_G$, we have $v - \omega \in \mathrm{Lie}(H)$, $\psi(\omega) = \omega$ and thus
	\begin{align*}
	\Vert \omega \Vert = \Vert \psi(\omega) \Vert 
	&\leq \Vert \psi(v) \Vert + \Vert \psi(v - \omega)\Vert \\
	&\leq \Vert v \Vert + \sum_{i=1}^{2g^\prime+t^\prime} |r_i-n_i| \cdot \Vert \omega_i \Vert \\
	&\leq \Vert v \Vert + c_3(2g^\prime+t^\prime)\deg_L(\overline{H}). \qedhere
	\end{align*}
\end{proof}

\subsection{Heights on Semiabelian Varieties}

In this subsection, we let $G$ be a semiabelian variety over a number field $\KK \subseteq \mathbb C$. Let $h_L$ be an arbitrary Weil height associated to $L$. In this general setting, the second author proved the following theorem, which constitutes a proof of the bounded height conjecture for semiabelian varieties. We write $G^{[s]}$ for the countable union of {all} algebraic subgroups having codimension $\geq s$ in $G$ for some fixed integer $s$. Recall that, for a subvariety $V\subseteq G$, we have defined  $V^{oa}$ right before Theorem \ref{ST}.

\begin{thm}[\cite{Kuehne2020}] \label{BHT} For any subvariety $V \subseteq G$ the height $h_L$ is bounded from above on the set $V^{oa}(\QQ)\cap G^{[\dim(V)]}(\QQ)$.
\end{thm}

%
%

We conclude this subsection with a further lemma.

\begin{lem}
\label{lemma::norm_estimate}
Let $\Vert \cdot \Vert$ be the norm on $\mathrm{Lie}(G)$ introduced in the previous subsection. Then, there exists a constant $c_5 = c_5(G) >0$ such that each $p \in G(\QQ)$ has a preimage $v \in \mathrm{Lie}(G)$ satisfying
\begin{equation*}
\Vert v \Vert \leq c_5 [\KK(p):\KK] \max \{1, h_L(p) \}.
\end{equation*}
\end{lem}

\begin{proof}
Without loss of generality, we can assume $\overline{G}=\overline{G}_0$ and $L=L_0$ where $\overline{G}_0$ and $L_0 = M \otimes \overline{\pi}^\ast N$ are as in the previous subsection. Since the line bundle $\overline{\pi}^\ast N$ is nef with empty base locus, there exists a constant $c_6=c_6(G,N)>0$ such that $$h_L(x) \geq h_M(x) - c_6$$ for all $x \in G(\QQ)$ (\cite[Theorem B.3.6 (b)]{Hindry2000}). Let $\hhat_M: \overline{G}(\QQ) \rightarrow \mathbb R$ be the canonical height associated with $M$ by means of the homogeneity relation $\overline{[2]}^\ast M = M^{\otimes 2}$. By \cite[Theorem B.4.1]{Hindry2000}, we have $$h_M(x) \geq \hhat_M(x) - c_7$$ for some constant $c_7=c_7(M)>0$. 

It is well known that the canonical height $\hhat_M$ decomposes into non-negative local heights such that the local height associated with the archimedean place encoded by the embedding $\mathbb K \subseteq \mathbb C$ equals $\sum_{i=1}^t |\lambda_i| : G(\mathbb C) \rightarrow \mathbb R$ (see e.g.~\cite[Proposition 4]{Bertrand1995}). Writing
\begin{equation*}
\mathcal{K}_{s} = \left\lbrace x \in G(\mathbb{C}): \sum_{i=1}^t |\lambda_i| \leq s \right\rbrace,
\end{equation*}
for some real number $s$, we conclude that $p \in \mathcal{K}_{s_0}$ with
\begin{equation*}
s_0 = [\KK(p):\KK] \hhat_M(p) \ll_{G,N} [\KK(p):\KK] \max \{1, h_L(p) \}.
\end{equation*}
Note that each $\mathcal{K}_s$, $s \geq 0$, is compact since it is a closed subset of the compact space $\overline{G}(\C)$.
Consequently, the preimage $\exp_G^{-1}(\mathcal{K}_1)$ is contained in a subset $\mathcal{K}^\prime + \Omega_G$ with $\mathcal{K}^\prime \subset \mathrm{Lie}(G)$ compact. Choose a point $p^\prime \in G(\QQ)$ such that $[\lceil s_0 \rceil](p^\prime)=p$. As each $\lambda_i$ is a homomorphism, we have $p^\prime \in \mathcal{K}_1$. By compactness, $p^\prime \in \mathcal{K}_1$ has a preimage $v^\prime \in \exp^{-1}(p^\prime)$ with $\Vert v' \Vert \ll_G 1$. Hence $v= \lceil s_0 \rceil \cdot v^\prime$ is the desired preimage of $p$.
\end{proof}

\subsection{Subgroups of semiabelian varieties}

We could not find a reference for the following lemma, but it should be well known to experts.

\begin{lem} \label{lemgalois1} 
	Let $G$ be a semiabelian variety over a field $k$. Then there exists a finite extension $k^\prime/k$ such that all connected algebraic subgroups of $G$ are defined over $k^\prime$ (i.e., all connected subgroups of $G$ are invariant under $\Gal(\overline{k}/k^\prime)$).
\end{lem}

Note that the assertion of the above lemma is false for the finite subgroups generated by torsion points, hence requiring connectedness is necessary.

\begin{proof}	
	We first consider the (well-known) case where $G$ is an abelian variety $A$. We start with proving that every connected algebraic subgroup $B \subseteq A$ appears as the component containing $0_A$ of the kernel of an endomorphism of $A$. In fact, there exists a connected algebraic subgroup $C \subseteq A$ such that $B + C = A$ and $B \cap C$ is finite. Write $\varphi: C/(B\cap C) \twoheadrightarrow C$ for the isogeny dual to the quotient map $C \twoheadrightarrow C/(B \cap C)=A/B$, $\iota: C \hookrightarrow A$ for the inclusion, and $\pi: A \twoheadrightarrow A/B$ for the quotient map. Then the kernel of $\iota \circ \varphi \circ \pi: A \rightarrow A$ has connected component $B$. 
	
	The endomorphism ring of $A_{\overline{k}}$ is a finitely generated $\mathbb{Z}$-module \cite[Theorem 3 on p.\ 176]{Mumford1970}, hence there exists a finite extension $k^\prime$ of $k$ such that every endomorphism of $A_{\overline{k}}$ is defined over $k^\prime$. By the above, any connected algebraic subgroup $B \subseteq A$ is the connected component containing $0_A$ of an algebraic subgroup $B^\prime \subseteq A$ invariant under $\Gal(\overline{k}/k^\prime)$ and hence likewise $\Gal(\overline{k}/k^\prime)$-invariant. This settles the case of abelian varieties.
	
	For a general semiabelian variety $G$, we can reduce to this case. Note that we can assume that the maximal subtorus of $G$ is split (i.e., equals $\Gm^t$) by replacing $k$ with a finite extension. By the above, we also assume that all connected algebraic subgroups of $A$ are defined over $k$. Under this assumption, we prove that the same is true for $G$. Let $H \subseteq G$ be a connected algebraic subgroup with maximal subtorus $T' \subseteq \Gm^t$ and maximal abelian quotient $B \subseteq A$. Note that $T'$ is split as a subtorus of a split torus so that we can arrange that $T'=\Gm^{t^\prime}$. Recall that a semiabelian variety $G$ defined over $k$ is described by an extension class $\eta_{G} \in \mathrm{Ext}^1_k(A,T)$ where $A$ is its maximal abelian quotient and $T$ is its maximal subtorus (see e.g.\ \cite[Subsections 1.1 and 1.2]{Kuehne2020} and \cite[Chapter VII]{Serre1988}). Furthermore, the Weil-Barsotti formula (see \cite[Section III.18]{Oort1966} or the appendix to \cite{Moret-Bailly1981}) gives a canonical identification $\mathrm{Ext}^1_k(A,\mathbb{G}_m) = A^\vee(k)$. This means that we can decompose $\eta_{G} = (\eta_{G,1},\cdots,\eta_{G,t})\in A^\vee(k)^{t}$ and $\eta_{H}=(\eta_{H,1},\dots,\eta_{H,t^\prime}) \in B^\vee(\QQ)^{t^\prime}$. Writing $\iota: B \hookrightarrow A$ for the inclusion and $\iota^\vee: A^\vee \rightarrow B^\vee$ for its dual, each $\eta_{H,i}$ is a $\mathbb Z$-linear combination of some $\iota^\vee(\eta_{G,i}) \in B^\vee(k)$, $1\leq i \leq t$. Consequently, all $\eta_{H,i}$, $1\leq i\leq t^\prime$, are contained in $A^\vee(k)$ and thus $H$ is defined over $k$.
\end{proof}

\subsection{An auxiliary lemma}

The following assertion generalizes \cite[Lemma 6]{BHMZ}.

\begin{lem} \label{lemgalois2} 
	Let $G$ be a semiabelian variety over a number field $\mathbb{K}$ such that all its connected subgroups are defined over $\KK$.
	
	\begin{itemize}
		\item[(a)] Let $p \in G(\QQ)$ and $\sigma \in \Gal(\overline{\mathbb{K}}/\mathbb{K})$ be such that $p^{\sigma} - p$ is contained in a coset $q + H$, $q \in G(\KK)$. Then there exists a torsion point $q^\prime \in G(\QQ)$ such that $p^{\sigma} - p \in q^\prime + H$ (i.e., $p^{\sigma} - p$ is contained in a torsion coset of $H$).
		\item[(b)] There exists an integer $e = e(\KK) \geq 1$ such that the following is true: If $p \in G(\QQ)$ satisfies $p^{\sigma} - p \in G(\mathbb{K})$ for all $\sigma \in \Gal(\overline{\mathbb{K}}/\mathbb{K})$, then $[e](p) \in G(\mathbb{K})$. 
	\end{itemize}
\end{lem}

\begin{proof} 	
	(a)	Let $\mathbb{L}$ be a finite normal extension of $\KK$ such that $p \in G(\mathbb{L})$. Consider the image $\chi(q) \in (G/H)(\mathbb{K})$ under the quotient homomorphism $\chi: G \rightarrow G/H$. We have to prove that $\chi(q) = \chi(p^\sigma - p)$ is torsion. In fact, its $[\mathbb{L}:\KK]$-th multiple equals
	\begin{equation*}
	\sum_{\tau \in \Gal(\mathbb{L}/\KK)} \chi(p^\sigma-p)^\tau =
	\sum_{\tau \in \Gal(\mathbb{L}/\KK)} \chi(p^{\tau \circ \sigma})-
	\sum_{\tau \in \Gal(\mathbb{L}/\KK)} \chi(p^\tau)
	= 0_{G/H}.
	\end{equation*}
	
	(b) The argument of (a), applied to the trivial group $H=1$, shows that $p^\sigma - p \in G(\KK)$ is a torsion point. We can take $e(\KK)$ to be the exponent of the finite group $\mathrm{Tors}(G) \cap G(\KK)$. Indeed, we have $[e](p^\sigma-p)=0$ for all $\sigma \in \Gal(\overline{\KK}/\KK)$, which is equivalent to $([e](p))^\sigma=[e]p$ for all $\sigma \in \Gal(\overline{\KK}/\KK)$. Thus, we have $[e]p \in G(\KK)$.
\end{proof}

\section{An auxiliary proposition}

In this section, we establish an important finiteness proposition needed in the course of our main proof. It generalizes Lemma 4 of \cite{BHMZ}, except for effectivity.

\begin{prop}\label{lemfink} 
	Let $G$ be a semiabelian variety defined over a number field $\mathbb{K}$ and $C \subseteq G$ an irreducible algebraic curve defined over $\QQ$ that does not lie in $G^{[1]}$ (i.e., is not contained in a proper subgroup of $G$). For any integer $e \geq 1$ and any number field $\mathbb{K}$, there are at most finitely many $p \in (C \cap G^{[2]})(\QQ)$ such that $[e](p) \in G(\mathbb{K})$. 
\end{prop}

Before being able to prove the proposition, we have to start with some preparatory lemmas. In the following, $G$ is a semiabelian variety defined over a number field $\KK$ and we assume that $G$ is the extension of the abelian variety $A$ by a split torus $\mathbb{G}^t_m$. Furthermore, $C \subseteq G$ is an irreducible subcurve defined over $\QQ$. We consider the compactification $\overline{G}$ of $G$ and the line bundles $M$, $N$ on $G$ as in Subsection \ref{subsection::degrees}. In addition, we recall the decomposition $$M = \bigotimes_{i=1}^t \left(M_{\overline{G}, i}^{(0)} \otimes M_{\overline{G}, i}^{(\infty)}\right)$$ from \cite[Section 2.2]{Kuehne2018} and set $M_i = M_{\overline{G}, i}^{(0)} \otimes M_{\overline{G}, i}^{(\infty)}$ for each $i \in \{1,\dots, t\}$. Finally, we write $\Sigma_{\mathbb{K}}$  for the set of places of $\mathbb{K}$ and $\Sigma_{\mathbb{K},f}$ for the finite ones. 

In the following lemma we use the Weil functions $\lambda_{i,\nu}$, for $i=1,\dots, t$ and $\nu \in \Sigma_{\mathbb{K}}$, given by \cite[Proposition 2.6]{Vojta1996}.

\begin{lem}
	\label{lemma::gap_filler}
	There exists a finite set of places $S \subseteq \Sigma_{\mathbb{K},f}$ and a finite set of rank $(t-1)$ matrices $\mathcal{M}^{(\alpha)} \in \mathbb{Z}^{(t-1) \times t}$, $\alpha \in \mathcal{I}$, with the following property: For each $\nu \in \Sigma_{\mathbb{K},f} \setminus S$ and $x \in C(\mathbb{K})$, there exists some $\alpha_0 = \alpha_0(\nu, x) \in \mathcal{I}$ such that
	\begin{equation}
	\label{equation::not_that_good_case}
	\mathcal{M}^{(\alpha_0)} \cdot 
	\begin{pmatrix} \lambda_{1,\nu}(x) \\ \vdots \\ \lambda_{t,\nu}(x) \end{pmatrix}
	=
	\begin{pmatrix} 0 \\ \vdots \\ 0 \end{pmatrix}.
	\end{equation}
\end{lem}

\begin{proof}
	We start with some geometric constructions. Let
	\begin{equation*}
	\eta_G = (P_1,\dots,P_t) \in \mathrm{Pic}_0(A)(\mathbb K)^t
	\end{equation*}
	be the extension class describing the semiabelian variety $G$. By \cite[Theorem 1 on p.\ 77]{Mumford1970}, there exist points $x_1,\dots,x_t \in A(\mathbb K)$ such that $P_i = T_{x_i}^\ast N \otimes N^{\otimes -1}$ ($1\leq i \leq t$) where $T_{x}: A \rightarrow A$ denotes the translation by $x \in A(\mathbb K)$. As $N$ is very ample and thus base-point free, there exist global sections $\mathbf{s}_k$, $1\leq k \leq k_0$, of $N$ such that the divisors $D_k=\mathrm{div}_N(\mathbf{s}_k)$ satisfy $\mathcal{O}(D_k) \approx N$, and
	\begin{equation}
	\label{equation::disjointness}
	\mathrm{supp}(D_1) \cap \mathrm{supp}(D_2) \cap \cdots \cap \mathrm{supp}(D_{k_0}) = \emptyset.
	\end{equation}
	We furthermore assume that $0_A, x_i \notin \mathrm{supp}(D_k)$ for all $i \in \{1,\dots, t\}$ and $k \in \{1,\dots,k_0\}$. Let us set $E_{k,k^\prime}^{(i)} = T_{x_i}^{\ast}D_k -D_{k^\prime}$ and notice that $$0_{A} \notin \mathrm{supp}(E_{k,k^\prime}^{(i)}) \subseteq (\mathrm{supp}(D_k) - x_i) \cup \mathrm{supp}(D_{k^\prime})$$ and 
	\begin{equation*}
	\mathcal{O}(E_{k,k^\prime}^{(i)}) \approx \mathcal{O}(T_{x_i}^{\ast}D_k) \otimes \mathcal{O}(-D_{k^\prime}) \approx T_{x_i}^\ast N \otimes N^{\otimes -1} \approx P_i
	\end{equation*}
	for all $i \in \{1,\dots,t\}$ and $k,k^\prime \in \{1,\dots,k_0 \}$.
	
	With each divisor $D_k$, $1\leq k \leq k_0$, and each place $\nu \in \Sigma_{\mathbb K}$, we can assign a Néron function $$\lambda_{D_k,\nu}\! : (A \setminus \mathrm{supp}(D_k))(\mathbb{C}_\nu) \longrightarrow \mathbb{R}$$ (compare \cite[Theorem 11.1.1]{Lang1983}); the Néron functions $(\lambda_{D_k,\nu})_{\nu \in \Sigma_{\mathbb K}}$ are only unique up to $\Sigma_{\mathbb{K}}$-constants and we make here a choice once and for all. By \cite[Corollary 10.3.3]{Lang1983}, the disjointness property \eqref{equation::disjointness}
	implies that there exists a finite set $S_0 \subseteq \Sigma_{\mathbb K,f}$ such that 
	\begin{equation}
	\label{equation::minimum_lambda}
	\min \{\lambda_{D_1,\nu}(x),\lambda_{D_2,\nu}(x), \dots, \lambda_{D_{k_0},\nu}(x) \} = 0
	\end{equation}
	for all $\nu \in  \Sigma_{\mathbb K,f}\setminus S_0$ and all $x \in A(\mathbb{C}_\nu)$.  By \cite[Theorem 11.1.1 (1) and (4)]{Lang1983}, the function
	\begin{equation*}
	\lambda_{E_{k,k^\prime}^{(i)},\nu} = (\lambda_{D_k,\nu} \circ T_{x_i} - \lambda_{D_{k^\prime},\nu}): (A \setminus \mathrm{supp}(E_{k,k^\prime}^{(i)}))(\mathbb{C}_\nu) \longrightarrow \mathbb{R}
	\end{equation*}
	is a Néron function for the divisor $E_{k,k^\prime}^{(i)} = T_{x_i}^{\ast}D_k -D_{k^\prime}$, $k,k^\prime \in \{1,\dots, k_0 \}$. Define the open subsets $U_{k,k^\prime}^{(i)} = A \setminus \mathrm{supp}(E_{k,k^\prime}^{(i)})$. By applying \eqref{equation::disjointness} twice, we obtain 
	\begin{align*}
		\bigcap_{1\leq k, k^\prime \leq k_0} \mathrm{supp}(E_{k,k^\prime}^{(i)}) 
		&\subseteq 
		\bigcap_{1\leq k, k^\prime \leq k_0} \left(\mathrm{supp}(D_k) - x_i) \cup \mathrm{supp}(D_{k^\prime})\right) \\
		&\subseteq
		\bigcap_{1\leq k \leq k_0} \left( \bigcap_{1\leq k^\prime \leq k_0} (\mathrm{supp}(D_k) - x_i) \cup \mathrm{supp}(D_{k^\prime}) \right) \\
		&\subseteq
		\bigcap_{1\leq k \leq k_0} \left(\mathrm{supp}(D_k) - x_i\right) \\
		&= \emptyset.
	\end{align*}
	In other words, $\bigcup_{1\leq k,k^\prime \leq k_0} U_{k,k^\prime}^{(i)} = A$ for all $1\leq i\leq t$. Furthermore, the restriction of the line bundle $P_i$ to $U_{k,k^\prime}^{(i)}$ is trivial, so that we can fix isomorphisms $\phi_{k,k^\prime}^{(i)}: P_i|_{U_{k,k^\prime}^{(i)}} \rightarrow U_{k,k^\prime}^{(i)} \times \mathbb{A}^1$. Projecting to the second factor induces ``toric coordinates'' $$z^{(i)}_{k,k^\prime}: \pi^{-1}(U_{k,k^\prime}^{(i)}) \rightarrow \mathbb{A}^1$$ ($1\leq i \leq t$, $1\leq k, k^\prime \leq k_0$). There exists a finite subset $S_1 \subseteq \Sigma_{\mathbb K, f}$ such that
	\begin{equation*}
	\log |z_{k,k^\prime}^{(i)}(0_G)|_\nu =0= \lambda_{E_{k,k^\prime}^{(i)},\nu}(0_A)
	\end{equation*}
	for all $\nu \in \Sigma_{\mathbb K, f} \setminus S_1$.\footnote{Indeed, $z_{k,k^\prime}^{(i)}(0_G)$ is an algebraic number so the first equality is clear. For the second equality, note that if $\lambda_{D,\nu}\!: (X \setminus \mathrm{supp}(D) )(\C_\nu) \rightarrow \R$, $\nu \in \Sigma_{\mathbb{K}}$, is the collection of Weil functions associated with an arbitrary Néron divisor $D$ on a $\mathbb{K}$-algebraic variety $X$, then for every point $x\in (X \setminus \mathrm{supp}(D))(\mathbb{K})$ we have $\lambda_{\nu}(x) = 0$ for almost all places $\nu$ of $\mathbb{K}$ (see \cite[Section 10.2]{Lang1983}). A closer look at the Néron-Tate limit process reveals that the same is then also true for the canonical heights on abelian varieties (see \cite[Equation (19)]{Call1993}).} From the definition of $\lambda_{i,\nu}$, $1\leq i \leq t$, in \cite[(2.6.3)]{Vojta1996}, we know that
	\begin{equation*}
	\lambda_{i,\nu}(x) = 
	\left(-\log(|z^{(i)}_{k,k^\prime}(x)|_\nu) + \lambda_{E_{k,k^\prime}^{(i)},\nu}(\pi(x))\right)
	-
	\left(-\log(|z^{(i)}_{k,k^\prime}(0_G)|_\nu) + \lambda_{E_{k,k^\prime}^{(i)},\nu}(0_A)\right)
	\end{equation*}
	for all $x \in \pi^{-1}(U^{(i)}_{k,k^\prime})(\mathbb C_\nu)$. Consequently, we have
	\begin{equation*}
	\lambda_{i,\nu}(x) = 
	-\log(|z^{(i)}_{k,k^\prime}(x)|_\nu) + \lambda_{E_{k,k^\prime}^{(i)},\nu}(\pi(x))
	\end{equation*}
	for all $x \in \pi^{-1}(U^{(i)}_{k,k^\prime})(\mathbb C_\nu)$ and $\nu \in \Sigma_{\mathbb K, f} \setminus S_1$.
	
	For every point $p \in C(\mathbb K)$, every $i \in \{1,\dots, t\}$ and every $\nu \in \Sigma_{\mathbb{K},f} \setminus S_0$, we can use \eqref{equation::minimum_lambda} to pick $k_i, k^\prime_i\in \{1, \dots, k_0\}$, which may depend on $p,i,\nu$, such that
	\begin{equation*}
	\lambda_{D_{k_i},\nu}(\pi(p) + x_i) = \lambda_{D_{k^\prime_i},\nu}(\pi(p)) = 0;
	\end{equation*}
	in particular, this means $\pi(p), \pi(p) + x_i \in U_{k_i,k_i^\prime}^{(i)}$. It follows that 
	\begin{equation*}
	\lambda_{E_{k,k^\prime}^{(i)},\nu}(\pi(p)) = \lambda_{D_{k_i},\nu}(\pi(p) + x_i )-  \lambda_{D_{k^\prime_i},\nu}(\pi(p)) = 0,
	\end{equation*}
	so that
	\begin{equation*}
	\lambda_{i,\nu}(p) = -\log(|z^{(i)}_{k_i,k^\prime_i}(p)|_{\nu})
	\end{equation*}
	for all $\nu \in \Sigma_{\mathbb{K},f}\setminus(S_0\cup S_1)$.
	Set $U=\pi^{-1}(U^{(1)}_{k_1,k_1^\prime}) \cap \cdots \cap \pi^{-1}(U^{(t)}_{k_t,k_t^\prime})$ and consider the map
	\begin{equation*}
	\varphi_p: 
	U \longrightarrow \mathbb{G}_m^t, \ x \longmapsto (z_{k_1,k_1^\prime}^{(1)}(x),z_{k_2,k_2^\prime}^{(2)}(x),\dots,z_{k_t,k_t^\prime}^{(t)}(x)).
	\end{equation*}
	Note that, although $\varphi_p$ depends on $p$, there are at most finitely many choices for this map. Therefore, there is a finite set $\{p_1,\dots,p_r\}\subseteq \mathbb{G}_m^t (\overline{\KK})$ of points and subcurves $C_1,\dots,C_s \subseteq \mathbb{G}_m^t$, independent of $p$ and $\nu_0$ such that the Zariski closure of $\varphi_p(C \cap U) \subseteq \mathbb{G}_m^t$ is among the elements of $\{p_1,\dots,p_r, C_1,\dots,C_s\}$. By Lemma \ref{lemma::monomial} below, we obtain a finite set $S_2 \subseteq \Sigma_{\mathbb{K},f}$ and matrices $\mathcal{M}^{(\alpha)} \in \mathbb{Z}^{(t-1) \times t}$ of rank $(t-1)$, $\alpha \in \mathcal{I}$, such that for each $x = (x_1,\dots,x_t) \in \varphi_p(C \cap U)$, there exists $\alpha_0 \in \mathcal{I}$ satisfying
	\begin{equation*}
	\mathcal{M}^{(\alpha_0)} \cdot 
	\begin{pmatrix} \log |x_1|_{\nu} \\ \vdots \\ \log |x_t|_{\nu} \end{pmatrix}
	=
	\begin{pmatrix} 0 \\ \vdots \\ 0 \end{pmatrix}.
	\end{equation*}
	for every $\nu \in \Sigma_{\mathbb{K},f} \setminus (S_0 \cup S_1 \cup S_2)$ and all $x \in \varphi_p(C \cap U)$. Applying this for $x=\varphi_p(p)$, we obtain the assertion of the lemma, as $\lambda_{i,\nu} = - \log |z_{k,k^\prime}^{(i)}|_\nu$.
\end{proof}

\begin{lem}
	\label{lemma::monomial}
	Let $C \subseteq \mathbb{G}_{m,\mathbb K}^t$ be an irreducible curve. There exist finitely many places $S \subseteq \Sigma_{\mathbb{K},f}$ and a finite set of rank $(t-1)$ matrices $\mathcal{M}^{(\alpha)} \in \mathbb{Z}^{(t-1) \times t}$, $\alpha \in \mathcal{I}$, with the following property: For each $\nu \in \Sigma_{\mathbb{K},f} \setminus S$ and $\underline{z} = (z_1,\dots,z_t) \in C(\mathbb{K})$, there exists some $\alpha_0 = \alpha_0(\nu, \underline{z}) \in \mathcal{I}$ such that
	\begin{equation}
	\label{equation::not_that_good_case2}
	\mathcal{M}^{(\alpha_0)} \cdot 
	\begin{pmatrix} \log|z_1|_\nu \\ \vdots \\ \log|z_t|_\nu \end{pmatrix}
	=
	\begin{pmatrix} 0 \\ \vdots \\ 0 \end{pmatrix}.
	\end{equation}
\end{lem} 
\begin{proof} Let us first consider the case $t=2$. (The case $t=1$ is evidently trivial.) The curve $C \subseteq \mathbb{G}_m^2$ can be written as the the zero set of a non-zero polynomial 
	\begin{equation*}
	F(X_1,X_2) = \sum_{i_1 = -n_1}^{n_1} \sum_{i_2=-n_2}^{n_2} a_{i_1,i_2} X^{i_1}_1 X^{i_2}_2 \in \mathbb{K}[X_1,X_2].
	\end{equation*} 
	Let $S \subseteq \Sigma_{\mathbb K, f}$ be a finite subset such that $|a_{i_1,i_2}|_\nu \in \{0,1\}$ for all $\nu \in \Sigma_{\mathbb K,f} \setminus S$ and all $i_1,i_2 \in \mathbb{Z}$.
	Fix $\underline{z} = (z_1,z_2) \in C(\mathbb{K})$ and $\nu \in \Sigma_{\mathbb K,f} \setminus S$. As
	\begin{equation*}
	F(z_1,z_2) = \sum_{i_1 = -n_1}^{n_1} \sum_{i_2=-n_2}^{n_2} a_{i_1,i_2} z^{i_1}_1 z^{i_2}_2 = 0,
	\end{equation*}
	there exist distinct pairs $(j_1,j_2)$ and $(k_1,k_2)$ such that
	\begin{equation*}
	|a_{j_1,j_2}|_\nu \neq 0, \ |a_{k_1,k_2}|_\nu \neq 0,
	\end{equation*}
	and
	\begin{equation*}
	|z_1|^{j_1}_\nu |z_2|^{j_2}_\nu = |a_{j_1,j_2} z^{j_1}_1 z^{j_2}_2|_\nu = |a_{k_1,k_2} z^{k_1}_1 z^{k_2}_2|_\nu = |z_1|_\nu^{k_1} |z_2|_\nu^{k_2}.
	\end{equation*}
	We infer a non-trivial relation
	\begin{equation*}
	(j_1-k_1) \cdot \log |z_1|_\nu + (j_2-k_2) \cdot \log |z_2|_\nu = 0.
	\end{equation*}
	Varying $z\in C(\mathbb{K})$, we obtain at most finitely many different equations of this form, one of which has to be satisfied for every point $z \in C(\mathbb{K})$ and every $\nu \in \Sigma_{\mathbb K, f}\setminus S$. Thus we obtain \eqref{equation::not_that_good_case2} in case $t=2$. 
	
	For the case $t>2$, we can assume that there exists an integer $t_0 \in \{1,\dots,t\}$, such that the projections $\mathrm{pr}_i|_C: C \rightarrow \mathbb{P}^1_\mathbb{K}$, $1\leq i \leq t_0$, are dominant, and each image $\mathrm{pr}_i(C)$, $t_0 + 1 \leq i \leq t$, is a point. We can enlarge $S \subseteq \Sigma_{\mathbb K, f}$ such that $ \log |z_{t^\prime}|_\nu = 0$, $t_0 + 1\leq t^\prime \leq t$, for all $\underline{z} \in C(\mathbb{K})$.
	
	Let $\underline{z} \in C(\mathbb{K})$ be an arbitrary point as in the assertion of the lemma. There is nothing to prove if $\log |z_i|_\nu = 0$ for all $i \in \{1,\dots,t_0 \}$. After renaming the first $t_0$ coordinates, we can therefore assume that $\log |z_1|_\nu \neq 0$. For each $t^\prime \in \{2,\dots, t\}$, applying the above to the projection $\mathrm{pr}_{1,t^\prime}(C) \subset \mathbb{G}_m^2$, we conclude that
	\begin{equation*}
	b_{1} \cdot \log |z_1|_\nu + b_{2} \cdot \log |z_{t^\prime}|_\nu = 0
	\end{equation*}
	for one of finitely many pairs $(b_1, b_2) \neq (0,0)$. In fact, we can and do assume that $b_2 \neq 0$ for each of these pairs.
	
	 For each point $\underline{z} \in C(\mathbb{K})$, we combine these $t-1$ relations to a matrix relation
	\begin{equation*}
	\mathcal{M} \cdot 
	\begin{pmatrix} \log|z_1|_\nu \\ \vdots \\ \log|z_t|_\nu \end{pmatrix}
	=
	\begin{pmatrix} 0 \\ \vdots \\ 0 \end{pmatrix}
	\end{equation*}
	where $\mathcal{M} \in \mathbb{Z}^{(t-1) \times t}$ has rank $(t-1)$. Clearly, we obtain only finitely many matrices $\mathcal{M}$ in this process, which concludes the proof of the lemma.
\end{proof}

\begin{lem}
	\label{lemma::mordell_lang}	
	For every finite subset $S \subseteq \Sigma_{\mathbb{K},f}$, the subgroup
	\begin{equation*}
	\Gamma(\mathbb{K},S) = \{ x \in G(\overline{\mathbb Q}) \ | \ \pi(x) \in A(\mathbb{K}) \ \text{and} \ \ \forall \nu \in \Sigma_{\mathbb{K},f} \setminus S: \lambda_{1,\nu}(x)=\cdots=\lambda_{t,\nu}(x)=0 \}
	\end{equation*}
	has finite rank.
\end{lem}

\begin{proof}
	By the Mordell-Weil theorem (see e.g.\ \cite[Appendix II]{Mumford1970}), the group $A(\mathbb K)$ is finitely generated. Let $\gamma_1,\dots,\gamma_n$ be generators of $A(\mathbb K)$. For each $\gamma_j$, $1\leq j \leq n$, we can find a preimage $\gamma^\prime_j \in \pi^{-1}(\gamma_j) \in G(\mathbb{K})$. Possibly enlarging $S$, we may additionally assume that $\lambda_{i,\nu}(\gamma^\prime_j)=0$ for all $\nu \in \Sigma_{\mathbb{K},f} \setminus S$, $1\leq i\leq t$ and $1\leq j \leq n$. Let additionally $\Gamma^\prime(S) \subseteq \mathbb{G}_m^t(\overline{\mathbb Q})$ be the $S$-units in the maximal torus of $G$. By Dirichlet's $S$-unit theorem \cite[Corollary I.11.7]{Neukirch1999}, the group $\Gamma^\prime(S)$ is finitely generated. As
	\begin{equation*}
	\Gamma(\mathbb{K},S) \subseteq \Gamma' (S)+ \mathbb{Z} \cdot \gamma_1^\prime + \cdots + \mathbb{Z} \cdot \gamma_n^\prime,
	\end{equation*}
	the group $\Gamma(\mathbb{K},S)$ is likewise finitely generated.
\end{proof}

With the preceding preparations, we can finally come back to the main result of this subsection.

\begin{proof}[Proof of Proposition \ref{lemfink}]
	Replacing $C$ with $[e](C)$, we may assume that $e=1$. This means that we have to show that the set $(C \cap G^{[2]})(\mathbb{K})$ is finite. We can also assume that $C$ is not a translate $H_0 + p$ where $H_0 \subseteq G$ is a subgroup (of dimension $1$) and $p \in G(\overline{\mathbb{Q}})$. In fact, if such a translate $H_0 + p$ intersects a subgroup $H \subseteq G$ of codimension $2$, then $H_0 + p \subseteq H_0+ H \subseteq G^{[1]}$. By Lemma \ref{lemma::mordell_lang} and the Mordell-Lang conjecture for semiabelian varieties that was proven by Faltings \cite{Faltings1994}, McQuillan \cite{McQuillan1995}, and Vojta \cite{Vojta1996}, the set $C(\mathbb{K}) \cap \Gamma(\mathbb{K},S)$ is finite for every finite subset $S \subseteq \Sigma_{\mathbb{K},f}$. Furthermore, the proposition follows from Faltings' proof of the Mordell conjecture \cite{Faltings1983} if $t=0$ (i.e., $G$ is an abelian variety).

	We prove the general case of the proposition by induction on $t$. For the inductive step, it remains to prove that $(C \cap G^{[2]})(\mathbb{K}) \setminus \Gamma(\mathbb{K},S)$ is finite. 
	
	Let $S \subseteq \Sigma_{\mathbb{K},f}$ and $\mathcal{M}^{(\alpha)} \in \mathbb{Z}^{(t-1) \times t}$, $\alpha \in \mathcal{I}$, be as in Lemma \ref{lemma::gap_filler} above. Each matrix $\mathcal{M}^{(\alpha)}$ determines a $1$-dimensional subgroup $H^{(\alpha)} \subseteq \mathbb{G}_m^t$ of the maximal torus and thus a quotient homomorphism $\varphi^{(\alpha)}: G \rightarrow G/H^{(\alpha)}$ to a semiabelian variety $G^{(\alpha)} = G/H^{(\alpha)}$. The image $\varphi^{(\alpha)}(C)$ is an irreducible algebraic subcurve of $G^{(\alpha)}$; for else $C$ is a translate of the subgroup $H^{(\alpha)}$. In addition, the curve $\varphi^{(\alpha)}(C)$ is not contained in $(G^{(\alpha)})^{[1]}$ or else $C \subseteq G^{[1]}$, which contradicts our assumptions. By our inductive assumption, we already know that the set $(\varphi^{(\alpha)}(C)\cap (G^{(\alpha)})^{[2]})(\KK) $ is finite.
	
	Consider now a subgroup $H \subseteq G$ of codimension $\geq 2$ containing a point $x \in C(\mathbb{K}) \setminus \Gamma(\mathbb{K},S)$. There exists some $\nu_0 \in \Sigma_{\mathbb{K},f} \setminus S$ such that $(\lambda_{1,\nu_0}(x), \dots, \lambda_{t,\nu_0}(x))$ is non-zero. Note that this implies that the subgroup $H$ has a non-trivial maximal subtorus $T \subseteq \mathbb{G}_m^{t}$; in fact, otherwise every point $x \in H(\mathbb{K})$ would satisfy $\lambda_{1,\nu_0}(x)= \cdots = \lambda_{t,\nu_0}(x)=0$. There exists some $\alpha_0 \in \mathcal{I}$ such that
	\begin{equation*}
	\mathcal{M}^{(\alpha_0)} \cdot 
	\begin{pmatrix} \lambda_{1,\nu_0}(x) \\ \vdots \\ \lambda_{t,\nu_0}(x) \end{pmatrix}
	=
	\begin{pmatrix} 0 \\ \vdots \\ 0 \end{pmatrix}.
	\end{equation*}
	As $\mathcal{M}^{(\alpha_0)}$ has maximal rank $t-1$, any $(a_1,\dots,a_t) \in \mathbb{Z}^t$ such that 
	\begin{equation*}
	(a_1,\dots,a_t)  \cdot \begin{pmatrix} \lambda_{1,\nu_0}(x) \\ \vdots \\ \lambda_{t,\nu_0}(x) \end{pmatrix} = 0
	\end{equation*}
	is a $\mathbb{Q}$-linear combination of rows in $\mathcal{M}^{(\alpha_0)}$. In particular, this is true for the equations describing the subtorus $T \subseteq \mathbb{G}_m^t$. We deduce that $H^{(\alpha_0)} \subseteq T \subseteq H$ so that $\varphi^{(\alpha_0)}(H)$ is a subgroup of codimension $\geq 2$ in $G^{(\alpha_0)}$. Thus, there are only finitely many possible choices for $\varphi^{(\alpha_0)}(x) \in \varphi^{(\alpha_0)}(C)(\KK) \cap (G^{(\alpha_0)})^{[2]}$ -- independent of the subgroup $H \subseteq G$. As $\varphi^{(\alpha_0)}|_C : C \rightarrow \varphi^{(\alpha_0)}(C)$ is finite, this leaves only finitely many possibilities for $x \in C(\KK)$. We conclude that $(C \cap G^{[2]}) (\mathbb{K})\setminus \Gamma(\mathbb{K},S)$ is finite.	
\end{proof}

\section{Unlikely intersections of bounded height}\label{boundedheight}

In this section, we prove the following intermediate result towards Theorem \ref{mainthm}. Its proof is based on o-minimal counting techniques and modeled after \cite{HP}. 

\begin{prop} \label{lemfinbh}
	Let $\cC$ be an irreducible curve in $G$ not contained in $G^{[1]}$. Then there are at most finitely many points in $ (\cC \cap G^{[2]})(\QQ)$ of bounded height. 
\end{prop}

In the sequel, $c_1,c_2,\dots$ denote positive constants that only depend on the semiabelian variety $G$ and on the curve $C$. We also assume throughout this section that the number field $\KK$ is sufficiently large such that all connected algebraic subgroups of $G$ are invariant under $\Gal(\overline{\KK}/\KK)$, which we may do by Lemma \ref{lemgalois1}.

\subsection{Some complexity estimates}

For each point $p \in G(\QQ)$, we write $\langle p\rangle=q +H_p $ for the smallest torsion coset of $G$ containing $p$, where $H_p$ is a connected algebraic subgroup of $G$ and $q \in \mathrm{Tors}(G)$. 

We define the \emph{complexity} of a torsion coset $q+H$ to be
$$
\Delta(q+H)=\max \{ \min\{\ord(q_0):q_0-q \in H,q_0 \in \mathrm{Tors}(G) \},\deg_L(\overline{H})\}
$$
for later use. We recall that $L$ is an ample line bundle on $\overline{G}$ that we have fixed before Lemma \ref{lemma::covolume_degree}.

The following two lemmas will allow us to bound $\Delta(\langle p\rangle)$ in terms of the degree of $p$ and its height.

\begin{lem}
\label{lemma::arithmetic_complexity}
	For each point $p \in G(\QQ)$, there is a torsion point $q \in \mathrm{Tors}(G)$ of order
	\begin{equation*}
		\ord(q) \leq c_1 [\KK(p):\KK]^{c_2}
	\end{equation*}
	such that $p \in q + H_p$.
\end{lem}
\begin{proof}
There is a commutative diagram of homomorphisms
	\begin{equation*} 
\begin{tikzcd}
0 \ar[r] & T \ar[r] \ar[d, hook] & H_p \ar[r] \ar[d, hook] & B \ar[d, hook] \ar[r] & 0 \\0 \ar[r] & \Gm^t \ar[r] \ar[d, ->>] & G \ar[r, "\pi_G"] \ar[d, "\varphi_{H_p}", ->>] & A  \ar[d, "\varphi_B", ->>] \ar[r] & 0 \\
0 \ar[r] & \Gm^t/T \ar[r] & G/H_p \ar[r, "\pi_{G/H_p}"] & A/B \ar[r] & 0
\end{tikzcd}
\end{equation*}
with exact rows and columns (compare \cite[Lemma 1]{Kuehne2020}). By definition, $\varphi_{H_p}(p)$ is a torsion point, and we have to bound its order in terms of $[\KK(p):\KK]$. By \cite[Lemma 9.3]{HP}, we know that $\varphi_B(\pi_G(p))$ is a torsion point of order 
\begin{equation*}
m_1 \leq c_3 [\KK(\pi_G(p)):\KK]^{c_4} \leq c_3 [\KK(p):\KK]^{c_4}.
\end{equation*}
Consider the point $p^\prime = [m_1](p) \in G(\QQ)$. Its image $\varphi_{H_p}(p^\prime)$ is a torsion point in $G/H_p$. Furthermore, it is contained in the maximal torus $\Gm^t/T$ of $G/H_p$ as
\begin{equation*}
\pi_{G/H_p}(\varphi_H(p^\prime))=\varphi_B(\pi_G(p^\prime))=[m_1](\varphi_B(\pi_G(p))) = 0_{A/B}.
\end{equation*} 
Being a quotient of a $\KK$-split torus $\Gm^t$, the torus $\Gm^t/T$ is also $\KK$-split (i.e., isomorphic over $\KK$ to some $\Gm^{t^\prime}$). Using the elementary structure of cyclotomic fields, we infer hence that the torsion point $\varphi_{H_p}(p^\prime)$ has order 
\begin{equation*}
m_2 \leq c_5 [\KK(\varphi_{H_p}(p^\prime)):\KK]^{c_6} \leq c_5 [\KK(p):\KK]^{c_6} .
\end{equation*}
As $\varphi_{H_p}(p)$ has order $\leq m_1 m_2$, this concludes the proof.
\end{proof}

\begin{lem}
\label{lemma::geometric_complexity}
For each point $p \in G(\QQ)$, we have
\begin{equation*}
\deg_L({H}_p) \leq c_7 [\KK(p):\KK]^{c_8} \max \{1, h_L(p)^{c_9} \}
\end{equation*}
\end{lem}

The assertion of the lemma is well known from the transcendence theory of commutative algebraic groups. For abelian varieties, it is proven in \cite[Théorème 1.4]{Bosser2019} -- with completely explicit constants $c_7, c_8, c_9$ -- relying on a theorem of Bertrand \cite[Théorème 2]{Bertrand1995}, polarization estimates due to Gaudron and Rémond \cite{Gaudron2014}, as well as other tools from transcendental number theory. Unfortunately, such a result does not seem to be in the literature for semiabelian varieties. Therefore, we derive the lemma directly from linear forms in logarithms of semiabelian varieties. In this way, the asserted degree bound follows from the bound on the degree of an obstruction subgroup.

\begin{proof}
We first fix a basis $\omega_1,\dots,\omega_{2g+t}$ of the period lattice $\Omega_G$.

Let $p \in G(\QQ)$. We recall that $\langle p \rangle= q+H_p $ for some $q\in \mathrm{Tors}(G)$ which implies $\langle p -q\rangle= H_p $. Using Lemma \ref{lemma::arithmetic_complexity} and easy estimates for the height (recall that torsion points have uniformly bounded height) we may replace $p$ by $p-q$ and therefore assume $\langle p \rangle= H_p $, so $p\in H_p(\QQ)$.

By the assumption on $\KK$, all connected algebraic subgroups of $G$ are defined over $\KK$. The tangent space $\mathrm{Lie}(H_p)$ is hence defined over $\KK$. Lemma \ref{lemma::norm_estimate} yields a preimage $v$ of $p$ under the group exponential $\exp_G$ such that 
\begin{equation*}
	\Vert v \Vert \ll_G [\KK(p):\KK] \max \{1, h_L(p) \}.
\end{equation*}
We do not necessarily have $v \in \mathrm{Lie}(H_p)$, but there exists a period $\omega = n_1 \omega_1 + \cdots + n_{2g+t} \omega_{2g+t} \in \Omega_G$ such that $v - \omega \in \mathrm{Lie}(H_p)$. We consider the homomorphism 
\begin{equation*}
	\psi: G \times G^{2g+t} \longrightarrow G, \ (x_0, x_1, \dots, x_{2g+t}) \longmapsto x_0 +[-n_1](x_1) + \cdots + [-n_{2g+t}](x_{2g+t}).
\end{equation*}
Our choice of $\psi$ is such that
\begin{equation*}
	\mathrm{Lie}(\psi)(v, \omega_1, \dots, \omega_{2g+t}) = v - n_1 \omega_1 - \cdots - n_{2g+t} \omega_{2g+t} = v - \omega \in \mathrm{Lie}(H_p).
\end{equation*}
We now let $U$ be the $\KK$-subspace $\mathrm{Lie}(\psi)^{-1}(\mathrm{Lie}(H_p))$ of $ \mathrm{Lie}(G \times G^{2g+t})$. Then, we can find $\KK$-subspaces $W_1,\dots, W_m\subseteq \mathrm{Lie}(G \times G^{2g+t})$ each of codimension 1 such that $U=\bigcap_{i=1}^{m}W_i$.

We apply  \cite[Théorème 1]{Gaudron2005} $m$ times with
\begin{enumerate}
	\item the semiabelian variety $\mathbf G$ therein being the product $G \times G^{2g+t}$,
	\item $\mathbf p$ being the point $(p,0_{G^{2g+t}}) \in \mathbf G(\KK)$,
	\item $\mathbf u$ being its logarithm $(v,\omega_1,\dots,\omega_{2g+t})$,
	\item $W$ being one of the codimension-1 $\KK$-subspaces $W_i$, $i=1,\dots, m$, and
	\item the real parameters $E, D, a$ being set such that
	\begin{equation*}
		E = e, \ D=[\KK:\mathbb Q], \ \log(a) = \max \left\{ 1, h_L(p), \frac{e^2(\Vert v \Vert^2 + \Vert \omega_1 \Vert^2 + \cdots + \Vert \omega_{2g+t} \Vert^2)}{D}\right\},
	\end{equation*}
\end{enumerate}
we obtain connected algebraic subgroups $\widetilde{\mathbf G}_i$ with $(p,0_{G^{2g+t}})  \in \widetilde{\mathbf G}_i(\KK) \subseteq (G \times G^{2g+t})(\KK)$ and
\begin{enumerate}
	\item $(v,\omega_1,\dots, \omega_{2g+t}) \in \mathrm{Lie}(\widetilde{\mathbf{G}}_i)$,
	\item $\mathrm{Lie}(\widetilde{\mathbf{G}}_i) \subseteq W_i$, and
	\item $\deg_{L^\prime}(\widetilde{\mathbf G}_i) \leq c_{10} [\KK(p):\KK]^{c_{11}} h_L(p)^{c_{12}}$ where $L^\prime = \pr_0^\ast L \otimes \cdots \otimes \pr_{2g+t}^\ast L$.
\end{enumerate}
Since $(p,0_{G^{2g+t}}) \in \bigcap_{i=1}^m \widetilde{\mathbf G}_i(\KK)$ and $H_p=\langle p\rangle$ we have $H_p\times \{ 0_{G^{2g+t}} \}\subseteq \bigcap_{i=1}^m \widetilde{\mathbf G}_i$. But (2) implies that $\bigcap_{i=1}^m\mathrm{Lie}(\widetilde{\mathbf{G}}_i) \subseteq \bigcap_{i=1}^m W_i=\mathrm{Lie}(\psi)^{-1}(\mathrm{Lie}(H_p))=U$ and thus
$$
\bigcap_{i=1}^m\mathrm{Lie}(\widetilde{\mathbf{G}}_i) \cap \mathrm{Lie}(G \times \{ 0_{G^{2g+t}} \})\subseteq  U\cap \mathrm{Lie}(G \times \{ 0_{G^{2g+t}} \})\subseteq \mathrm{Lie}(H_p \times \{ 0_{G^{2g+t}} \})
$$
Thus, $H_p \times \{0_{G^{2g+t}}\}$ is the identity component of the algebraic subgroup $\bigcap_{i=1}^m\widetilde{\mathbf G}_i \cap (G \times \{0_{G^{2g+t}}\})$. Using Bézout's Theorem (see e.g.\ \cite[Corollary 2.26]{Vogel}) for the Segre embedding, we can hence deduce that
\begin{align*}
	\deg_L(H_p) 
	&= \deg_{L^\prime}(H_p \times \{ 0_{G^{2g+t}}\}) \\
	&\leq  \prod_{i=1}^m \deg_{L^\prime}(\widetilde{\mathbf G}_i) \deg_{L^\prime}(G \times \{ 0_{G^{2g+t}} \}) \\
	&\ll_G [\KK(p):\KK]^{c_8}\max \{1, h_L(p) \}^{c_9}. \qedhere
\end{align*}
\end{proof}

By Lemma \ref{lemma::arithmetic_complexity} and \ref{lemma::geometric_complexity}, we have 
\begin{equation}
\label{lemma::complexity}
	\Delta(\langle p \rangle) \leq c_{13} [\KK(p):\KK]^{c_{14}} \max \{1, h_L(p)^{c_{15}} \}
\end{equation}
for each point $p \in G(\QQ)$.

\subsection{Definability of the exponential map}\label{Subs:Def}

The following lemma prepares our application of o-minimal counting techniques, for which we need that a suitable restriction of the exponential map $\exp_G: \mathrm{Lie}(G) \rightarrow G(\mathbb{C})$ is definable in the o-minimal structure $\mathbb{R}_{\mathrm{an},\mathrm{exp}}$ \cite{Dries1994}. To make this precise, we need to fix some additional notations and identifications first. 

A definable manifold is a pair $(M,\{\varphi_i: U_i \rightarrow \mathbb R^n\}_{1\leq i \leq K})$ consisting of a real-analytic manifold $M$ and a collection $\{\varphi_i: U_i \rightarrow \mathbb R^n\}_{1\leq i \leq K}$ of finitely many real-analytic charts covering $M$ such that the sets $\varphi_i(U_i \cap U_j)$ ($1 \leq i,j \leq K$) and the transition maps
\begin{equation*}
\varphi_j \circ \varphi_i^{-1}: \varphi_i(U_i \cap U_j) \rightarrow \varphi_j(U_i \cap U_j), \ 1 \leq i,j \leq K,
\end{equation*}
are definable in $\mathbb{R}_{\mathrm{an},\mathrm{exp}}$. In this situation, a subset $X \subseteq M$ is called definable if every $\varphi_i(U_i \cap X) \subseteq \mathbb R^n$ is definable. A map $f: M \rightarrow N$ between definable manifolds is called definable if the associated graph manifold is definable as a subset of the definable manifold $M \times N$. 

We endow $\mathrm{Lie}(G)$ with the structure of a definable manifold by taking a fixed $\mathbb{R}$-linear isomorphism $\iota: \mathrm{Lie}(G) \rightarrow \mathbb{R}^{2(g+t)}$ such that $\iota(\Omega_G) \subseteq \mathbb{Q}^{2(g+t)}$ as a (global) chart. Let $\mathcal{F}_G^\prime \subset V_G$ be a fundamental parallelepiped of the lattice $\Omega_G \subset V_G$. Recalling the decomposition $I V_{\Gm^t} \times V_{G}$ from \eqref{equation::lie_decomposition}, we set $\mathcal{F}_G = IV_{\Gm^t} \times \mathcal{F}^\prime_G$. It is easy to see that $\mathcal{F}_G^\prime$ and $\mathcal{F}_G$ are definable (as subsets of the definable manifold $\mathrm{Lie}(G)$); they are also canonically definable manifolds.

Furthermore, we can choose a projective embedding $\kappa: G \hookrightarrow \mathbb{P}^N_{\mathbb K}$ associated with the global sections of the very ample line bundle $L$. To endow the real-manifold $G(\mathbb{C})$ with the structure of a definable manifold, we use the $(N+1)$ charts induced from the standard covering of $\mathbb{P}^N_{\mathbb K}$ by open affine subsets. 

With these preparations, we can finally state the next lemma.

\begin{lem}\label{lemdef}
	Considering $\mathcal{F}_G$ and $G(\mathbb{C})$ as definable manifolds in the sense above, the restriction $\exp_G|_{\mathcal{F}_G}: \mathcal{F}_G \rightarrow G(\mathbb{C})$ is definable.
\end{lem}

\begin{proof} We use again the decomposition $\mathrm{Lie}(G) = I V_{\Gm^t} \times V_{G}$. For each $x \in I V_{\Gm^t}$ and $y \in V_{G}$, we have $\exp_G(x+y)=\exp_G(x) + \exp_G(y)$ as $G$ is commutative. As the group law $G(\mathbb C) \times G(\mathbb C) \rightarrow G(\mathbb C)$ is algebraic and hence definable, it suffices to prove that the restrictions of $\exp_G$ to $IV_{\Gm^t}$ and to $\mathcal{F}_G^\prime$ are definable. By compactness, the restriction of $\exp_G$ to $\mathcal{F}_G^\prime$ is definable (even in $\mathbb R_{\mathrm{an}}$). 
	
	It hence remains to prove the definability of $\exp_G|_{IV_{\Gm^t}}$ and using the fact that the group law is algebraic once again, we can even assume $t=1$ without loss of generality. In this situation, the identification $\Omega_{\Gm} = \mathbb Z \cdot (2\pi i)$ yields $IV_{\mathbb{G}_m} = \mathbb R$ and $\exp_G|_{IV_{\Gm}} = \exp(c \cdot x)$ for some real constant $c$. In any case, we see that $\exp_G|_{IV_{\Gm^t}}$ is definable.
\end{proof}

\subsection{O-minimal counting}


Recall that Lemma \ref{lemdef} provides us with a fundamental domain $\mathcal{F}_G\subseteq I V_{\Gm^t} \times V_{G}$ so that $\exp_G|_{\mathcal{F}_G}: \mathcal{F}_G \rightarrow G(\mathbb{C})$ is definable in the o-minimal structure $\mathbb{R}_{\mathrm{an},\mathrm{exp}}$. 

We are also going to identify $I V_{\Gm^t} \times V_{G}$ with $\R^{2t+2g}$ by choosing a basis so that the period lattice $\Omega_G$ corresponds to $\Z^{t+2g}$ in $ V_{G}$.

We let 
$$
\log_G(\cC)=(\exp_G|_{\mathcal{F}_G})^{-1}(\cC(
\C)).
$$
This is a definable set of dimension 2 (see \cite[Lemma 6.2]{HP}).

We have an induced embedding of $\mathrm{End}(\mathrm{Lie}(G))$ in $M_{2t+2g}(\R)$, which we identify with $\R^{(2t+2g)^2}$.
This will allow us to see each $\mathrm{Lie}(H)\subseteq \mathrm{Lie}(G)$ as the kernel of a matrix.

We consider the definable set
$$
Z=\{(\psi, x,z) \in M_{2t+2g}(\R) \times\R^{2t+2g} \times \R^{2t+2g} : z \in \log_G(\cC), \psi (z-x)=0 \}.
$$
We see it as a family with parameters in $M_{2t+2g}(\R)$ and fibers
$$
Z_{\psi_0}=\{( x,z) \in \R^{2t+2g} \times \R^{2t+2g}: z \in \log_G(\cC), \psi_0 (z-x)=0 \}.
$$

We moreover set, for some $T\geq 1$,
$$
Z_{\psi_0}(\Q,T)=\{( x,z)\in Z_{\psi_0}: x\in  \Q^{2t+2g} , H(x)\leq T \},
$$
where $H(\cdot)$ is the exponential height on $ \Q^{2t+2g}$.

Finally, we let $\pi_1$ and $\pi_2$ be the projection maps from $\R^{2t+2g} \times \R^{2t+2g}$ to the first and the second factor respectively.

The following statement is a special case of \cite[Corollary 7.2]{HP}.

\begin{lem}\label{lem::o-min} For every $\epsilon>0$ there exists a constant $c=c(Z,\epsilon)>0$ that satisfies the following property. If $T\geq 1$ and $\Sigma \subseteq Z_{\psi_0}(\Q,T)$ with $|\pi_2(\Sigma)| >cT^\epsilon$, there exists a continuous and definable function $\beta:[0,1]\to Z_{\psi_0}$ such that
	\begin{enumerate}
		\item the composition $\pi_1\circ \beta$ is semialgebraic and its restriction to $(0,1)$ is real analytic;
		\item the composition $\pi_2\circ \beta$ is non-constant;
		\item we have $\pi_2(\beta(0))\in \pi_2(\Sigma)$.
	\end{enumerate}
\end{lem}

We are also going to need the following consequence of Ax's Theorem \cite{Ax}.

\begin{lem} \label{lem::Ax}
	Let $\gamma:[0,1]\to \mathrm{Lie}(G)$ be real semialgebraic and continuous with $\gamma|_{(0,1)}$ real analytic. The Zariski closure in $G$ of the image of $\exp_G \circ \gamma$ is a coset of $G$.
\end{lem}

\begin{proof}
In \cite[Theorem 5.4]{HP} Habegger and Pila formulated and proved this statement for abelian varieties. The exact same proof works in our case as Ax's Theorem holds for semiabelian varieties.
\end{proof}

\subsection{Conclusion} 
In order to prove Proposition \ref{lemfinbh} we suppose there is a real number $B$ and infinitely many points $p\in (\cC\cap G^{[2]})(\QQ)$ with $\hhat_L(p)\leq B$. Note that, by Northcott's Theorem, we have that the degree over $\KK$ of such points must tend to infinity.

Let $p$ be one of these points. Then, $p\in \langle p \rangle =H_p+q$ for some $q \in \mathrm{Tors}(G)$ and $\dim H_p\leq \dim G-2$.  We assume that $q$ is of minimal order.

By our assumption on $\KK$ we have that all $\Gal(\overline{\KK}/\KK)$-conjugates of $p$ lie in an algebraic subgroup of $G$ of codimension at least 2. Actually we have $\langle p^\sigma\rangle= H_p+q^\sigma$ for all $\sigma\in \Gal(\overline{\KK}/\KK)$ and thus $\Delta( \langle p^\sigma \rangle)=\Delta( \langle p \rangle)$.

We let $z_\sigma, r_\sigma \in \log_G(\cC)$ be the logarithms of $p^\sigma$ and $q^\sigma$ in our fundamental domain, i.e., $\exp_G(z_\sigma)=p^\sigma$ and $\exp_G(r_\sigma)=q^\sigma$. Since we have identified $\Omega_G$ with $\Z^{t+2g} $ in $V_G$ and $r_\sigma\in \mathcal{F}_G$, we have that $r_\sigma\in \{0\} \times \Q^{t+2g}$ and $H(r_\sigma)=\ord(q^\sigma)$.

Now, Lemma \ref{lemma::norm_estimate} provides us with a $v_\sigma \in \Omega_G+\mathrm{Lie}(H_p)$ such that $\exp_{G}(v_\sigma)=p^\sigma-q^\sigma$ and 
$$
\Vert v_\sigma\Vert \leq c_{16}[\KK(p^\sigma,q^\sigma):\KK] \max \{1,h_L(p^\sigma)\}\leq c_{17}[\KK(p):\KK] \Delta(\langle p \rangle),
$$
as $h_L(p)\leq B$.

We compare $\Vert v_\sigma\Vert $ and $\Vert z_\sigma-r_\sigma\Vert$. Since their image via $\exp_{G}$ coincide, the difference between $v_\sigma$ and $z_\sigma-r_\sigma$ is a period, which must project to the identity on $IV_{\mathbb{G}_m^t}$. Therefore, the projections of these two elements on $IV_{\mathbb{G}_m^t}$ coincide. Moreover, the projection of $z_\sigma-r_\sigma$ on $V_G$ is bounded, therefore 
$$
\Vert z_\sigma-r_\sigma\Vert \leq c_{18} \Vert v_\sigma\Vert \leq c_{19}[\KK(p):\KK] \Delta(\langle p \rangle).
$$
%
%

Moreover, Lemma \ref{lemma::minkowski}(2) guarantees that there exists $\omega_\sigma \in \Omega_G$  with $z_\sigma-(r_\sigma+\omega_\sigma)\in \mathrm{Lie}(H_p)$ and 
$$
\Vert \omega_\sigma \Vert \leq \Vert z_\sigma-r_\sigma \Vert +c_{20} \deg_L(H_p)\leq c_{21} [\KK(p):\KK]  \Delta(\langle p \rangle).
$$

We conclude
$$
H(\omega_\sigma+r_\sigma) \leq c_{22}[\KK(p):\KK] \Delta(\langle p \rangle),
$$
and therefore we have
$$
H(\omega_\sigma+r_\sigma)  \leq c_{23} [\KK(p):\KK]^{c_{24}}
$$
by \eqref{lemma::complexity}. Set $d_p:=[\KK(p):\KK]$. Then,
$$
\Sigma_p:=\{(r_\sigma+\omega_\sigma,z_\sigma):\sigma\in \Gal(\overline{\KK}/\KK) \}  \subseteq Z_{\psi_0}(\Q,c_{23} d_p^{c_{24}}) ,
$$
where $\psi_0$ is a matrix whose kernel is $\mathrm{Lie}(H_p)$.

We now apply Lemma \ref{lem::o-min} with $\epsilon=1/(2c_{24})$, $T=c_{23} d_p^{c_{24}}$ and 
$\Sigma=\Sigma_p$. For $d_p$ large enough, we have $c(c_{23} d_p^{c_{24}})^\epsilon<d_p$ and therefore
$$
|\pi_2(\Sigma_p)|=d_p>c T^\epsilon.
$$
Lemma \ref{lem::o-min} then ensures the existence of a continuous and definable function $\beta:[0,1]\to Z_{\psi_0}$ satisfying 
	\begin{enumerate}
	\item the composition $\pi_1\circ \beta$ is semialgebraic and its restriction to $(0,1)$ is real analytic;
	\item the composition $\pi_2\circ \beta$ is non-constant;
	\item we have $\pi_2(\beta(0))\in \pi_2(\Sigma)$.
\end{enumerate}

We now consider the quotient $\phi:G\to G/H_p$ and the corresponding $\mathrm{Lie}(\phi):\mathrm{Lie}(G)\to \mathrm{Lie}(G/H_p)$ whose kernel is the kernel of $\psi_0$.

We note that by definition $$\exp_{G/H_p}\circ \Lie(\phi)\circ\pi_1\circ \beta=\exp_{G/H_p}\circ\Lie(\phi)\circ \pi_2 \circ\beta=\phi \circ {\exp_G} \circ {\pi_2} \circ \beta.$$
We apply Lemma \ref{lem::Ax} with $\gamma=\mathrm{Lie}(\phi)\circ\pi_1\circ\beta$. The Zariski closure of the image of ${\exp_{G/H_p}}\circ\gamma$ is a coset of $G/H_p$. On the other hand, the Zariski closure of the image of ${\exp_{G}}\circ\pi_2\circ\beta$ is contained in $C$. This containment cannot be strict because $\pi_2\circ \beta$ is non-constant.

Therefore, the Zariski closure of the image of $\phi \circ{\exp_{G}}\circ{\pi_2}\circ\beta$ equals $\phi(\cC)$ and is a one dimensional coset that must be a torsion coset because it contains the torsion point $\phi(\pi_2(\beta(0)))$.

Since $G/H_p$ has dimension at least 2 it follows that $\cC$ is contained in a proper algebraic subgroup of $G$. This contradicts our hypothesis and therefore $d_p$ must be uniformly bounded and Proposition \ref{lemfinbh} is proved.

\section{Proof of Theorem \ref{mainthm}}

\subsection{Preparations}

We recall that $G$ is a semiabelian variety defined over a number field $\mathbb{K}$, which is given by an exact sequence
\begin{equation*}
\begin{tikzcd}
1 \arrow{r} & T \arrow{r} & G \arrow{r}{\pi} &A \arrow{r} &0.
\end{tikzcd}
\end{equation*}

We recall that $L$ is an ample line bundle on $\overline{G}$ that we have fixed before Lemma \ref{lemma::covolume_degree}.

We prove the theorem by induction on $\dim(G)$. If $\dim(G)=1$, then $G^{[2]}=\emptyset$ and there is nothing to prove. If $\dim(G)=2$, then $G^{[2]}= \mathrm{Tors}(G)$ and the theorem is a consequence of the Manin-Mumford conjecture for semiabelian varieties proven by Hindry \cite{Hindry1988}. These two cases serve as the basis of our induction.

{For the induction step, we assume now that the theorem is already proven for all semiabelian varieties of dimension strictly less than $\dim(G) > 2$.} Let $p_0 + G_0$, $p_0 \in \cC(\overline{\Q})$, be the smallest coset containing $\cC$. Although $\cC$ is by assumption not contained in any proper algebraic {subgroup} of $G$, we may have $G_0 \neq G$.

We make an elementary observation on the dimension of $G_0$.

\begin{lem}\label{lem::mainproof1} Let $H \subset G$ be a semiabelian subvariety of codimension at least $2$ and $q \in \mathrm{Tors}(G)$. If the intersection $\cC \cap (q + H)$ is non-empty, then $G_0 + H=G$. In particular, it holds that $\dim(G_0) \geq 2$ if $\cC \cap G^{[2]}$ is non-empty.
\end{lem}

\begin{proof} Note that $G_0 + H$ is a connected algebraic subgroup of $G$. Choose a point $p \in (\cC \cap (q + H))(\QQ)$. As $\cC \subseteq p + G_0 \subseteq q + G_0 + H$, we have $G=G_0 + H$ because $\cC$ is not contained in a proper torsion coset. 
\end{proof}

{In the sequel, we hence assume that $\dim(G_0)\geq 2$.}





\subsection{An auxiliary surface}

We consider the difference map
$$
\begin{array}{llll}
\Delta: & G\times G & \longrightarrow & G, \\
&(p_1,p_2)& \longmapsto & p_1 - p_2,
\end{array}
$$
and the irreducible variety $S $ that is the Zariski-closure of $\Delta(\cC \times \cC) = \cC - \cC$ in $G$. 

It is easy to see that $G_0$ is the {minimal coset} containing $S$: In fact, any coset containing $S$ is a connected algebraic subgroup because of $0_G \in S$. If $H \subseteq G$ is a subgroup containing $S$, we have $$\cC \subset (p_0 + G_0) \cap (p_0 + H) = p_0 + (G_0 \cap H)$$
and thus $G_0 \subseteq G_0 \cap H$ as $G_0$ is the minimal subgroup that contains a translate of $\cC$.

In the next lemma, we collect some basic properties of $S$.

\begin{lem} 
\label{lemma::surface_nonanomalous}	
$S$ has dimension $2$ and is not $G_0$-anomalous in itself.
\end{lem}

\begin{proof}
For the sake of contradiction, assume that $S = \cC -\cC$ is an irreducible curve. Then, the translated curve $\cC^\prime = \cC - p_0$ satisfies $\cC^\prime - \cC^\prime \subseteq \cC^\prime$ and is hence a one-dimensional algebraic subgroup. As $\cC$ is just a translate of this subgroup, this gives a contradiction to our assumption that $\dim(G_0) \geq 2$. 

The second part of the statement was proven above.
\end{proof}

The Structure Theorem \ref{structure} applied to $S$ in $G_0$ gives a finite set $\Phi_S$ of proper abelian subvarieties of $G_0$ such that 
$$
S \setminus S^{\mathrm{oa}} = \bigcup_{H\in \Phi_S} \mathscr{Z}_H,
$$
where each $\mathscr{Z}_H$ is a finite union of subvarieties $W\subset S\cap (p+H )$, for some $p\in G_0(\QQ)$, with $\dim W \geq 1$ and $\dim W \geq \dim S - \codim H +1=3-\codim H$. Note that Lemma \ref{lemma::surface_nonanomalous}, implies that all the $W$ in the finite union above are curves. Moreover, all the $H$ that give a contribution to the above union have codimension at least 2.

We may then conclude that there exist finitely many irreducible algebraic curves $C_i \subset S$, $1 \leq i \leq N$, such that
\begin{equation}
\label{equation::Soa}
S \setminus S^{\mathrm{oa}} = C_1 \cup \cdots \cup C_{N}.
\end{equation}
Each curve $C_i$, $1 \leq i \leq N$, is contained in a coset $p_i + H_i \subset G_0$ with $p_i \in C_i(\QQ)$ and $H_i \subset G_0$ a subgroup of codimension at least $2$; for the sequel, we stipulate that each coset $p_i + H_i$ is the minimal coset containing $C_i$.

\subsection{A height comparison}

The restriction $\Delta|_{\cC \times \cC}: \cC \times \cC \rightarrow S$ is a dominant, generically finite map. The following standard lemma allows us to bound the height of most $\QQ$-points on $\cC \times \cC$ by the height of their images in $S$.

\begin{prop}
	\label{lemma::heightbound}
	Let $X$ (resp.\ $Y$) be an irreducible algebraic variety defined over $\QQ$, $L$ (resp.\ $M$) an ample line bundle on $X$ (resp.\ $Y$). Assume furthermore that $\dim X=\dim Y$ and let $f:X\rightarrow Y$ be a dominant morphism over $\QQ$. Set
	\begin{equation*}
	Z(f):= \{ y \in Y: \text{the fiber $X_y$ of $f$ over $y$ has dimension $>0$} \}.
	\end{equation*}
	Then there exist constants $c_1, c_2 > 0$ such that
	\begin{align}
	\label{equation::height_inequality}
	h_L(x) \leq c_1 h_M(f(x)) + c_2
	\end{align} 
	for all points $x \in X(\QQ)$ with $f(x) \notin Z(f)$. 
\end{prop}

By upper semicontinuity of the fiber dimension (\cite[Th\'eor\`eme 13.1.3]{EGA4III}), the set $Z(f)$ is closed. In addition, it is easy to show that $Z(f)$ has codimension $\geq 2$ in $Y$. In our application to surfaces $X$ and $Y$, this means that $Z(f)$ is a finite set of points.

\begin{proof}
		This follows from \cite[Theorem 1]{Silverman2011}. The height inequality obtained in \textit{loc.cit.} is of the desired form \eqref{equation::height_inequality}, but only valid for all points in a (not determined) Zariski open dense subset $U \subseteq Y$; we hence have to verify that we can ensure $Y \setminus U \subseteq Z(f)$.
		
		Let $W$ be an irreducible component of $Y \setminus U$ that does not lie in $Z(f)$. We have that $f^{-1}(W)$ consists of $m$ irreducible components $V_1, \dots ,V_m$. If $\dim(V_i)>\dim(W)$, then $f(V_i) \subseteq Z(f)$. So it suffices to consider points on the irreducible components $V_i$ such that $\dim(V_i)=\dim(W)$. For each of these, we can again apply  \cite[Theorem 1]{Silverman2011} to $f|_{V_i}:V_i\to W$. We get that \eqref{equation::height_inequality} holds for all $x\in V_i(\QQ)$ such that $f(x)\in U'$ for some Zariski open dense $U'\subseteq W$. If $W\setminus U'\subseteq Z(f)$ we are done, otherwise we repeat the same argument with $W\setminus U'$ instead of $W$. We only need to do this finitely many times as $\dim (W\setminus U')<\dim (W)$.
\end{proof}

Applying this lemma to the restriction $\Delta|_{\cC \times \cC}: \cC \times \cC \rightarrow S$, we obtain a finite set of $\QQ$-rational points $Z \subset S(\QQ)$ and constants $c_1,c_2>0$ such that
\begin{equation}
\label{equation::heightbound1}
h_L(p) + h_L(p') \leq c_1 h_L(p-p') + c_2
\end{equation}
whenever $p-p' \notin Z$ for points $p,p' \in C(\QQ)$.

\subsection{Enlarging the number field $\KK$}

There exists a finite extension $\KK^\prime/\KK$ such that the following conditions are satisfied:

\begin{enumerate}
	\item all connected algebraic subgroups of $G$ are defined over $\KK^\prime$ (Lemma \ref{lemgalois1}),
	\item we have $p_0 \in G(\KK^\prime)$, $p_i \in G(\KK^\prime)$ ($1\leq i \leq N$), and $Z \subset S(\KK^\prime)$, and
	\item the curve $\cC \subset G$ as well as the curves $C_i \subset G_0$ ($1\leq i \leq N$) in (\ref{equation::Soa}), are defined over $\KK^\prime$.
\end{enumerate}

{Replacing $\KK$ with $\KK^\prime$, we can assume that the above conditions are already satisfied for $\KK$.}

\subsection{Some reductions}

Recall that we have to show that the set $(\cC \cap G^{[2]})(\QQ)$ is finite. For this purpose, we consider a point $p\in (\cC \cap G^{[2]})(\QQ)$. Denote by $H \subset G$ a semiabelian subvariety of codimension at least $2$ and $q \in \mathrm{Tors}(G)$ such that $p \in q + H$. Recall that $G_0 + H=G$ by Lemma \ref{lem::mainproof1}.

Set $p^\prime = p - p_0 \in S(\QQ)$. For each $\sigma \in \Gal(\overline{\mathbb{K}}/\mathbb{K})$, we define
\begin{equation*}
r_{p,\sigma} = (p^\prime)^\sigma - p^\prime = p^\sigma - p \in (G_0 \cap \left(q^\sigma - q + H\right))(\QQ).
\end{equation*}
Considering tangent spaces at $0_G$, we obtain
\begin{equation*}
\dim(G_0)-\dim(G_0 \cap H) = \dim(G_0 + H) - \dim(H) = \dim(G)-\dim(H) \geq 2.
\end{equation*}
As the intersection $G_0 \cap (q^\sigma - q + H)$ is a torsion translate of $G_0 \cap H$ in $G_0$, the point $r_{p,\sigma}$ is contained in $S \cap (G_0)^{[2]}$.

\begin{lem}
\label{lemma::mainproof1}
There are at most finitely many points $p\in (\cC \cap G^{[2]})(\QQ)$ such that $r_{p,\sigma} \in Z$ for all $\sigma \in \Gal(\overline{\KK}/\KK)$.
\end{lem}

\begin{proof}
Assume that $p \in (\cC \cap G^{[2]})(\QQ)$ is such that $r_{p,\sigma} = p^\sigma - p \in Z \subset G(\KK)$ for all $\sigma \in \Gal(\overline{\KK}/\KK)$. By Lemma \ref{lemgalois2} (b), there exists an integer $e = e(\KK) \geq 1$, which is independent of $p$, such that $[e](p) \in G(\KK)$. By Proposition \ref{lemfink}, there are at most finitely many $p \in (\cC \cap G^{[2]})(\QQ)$ with this property.
\end{proof}

\begin{lem}
\label{lemma::mainproof2}
There are at most finitely many points $p\in (\cC \cap G^{[2]})(\QQ)$ such that there exists an automorphism $\sigma \in \Gal(\overline{\KK}/\KK)$ with $r_{p,\sigma} \in S^{\mathrm{oa}}(\QQ) \setminus Z$.
\end{lem}

\begin{proof}
Let $p \in(\cC \cap G^{[2]})(\QQ)$ and $\sigma \in \Gal(\overline{\KK}/\KK)$ be such that $r_{p,\sigma} \in S^{\mathrm{oa}} (\QQ)\setminus Z$. Recall that we have also $r_{p,\sigma} \in (G_0)^{[2]}(\QQ)$ as noted above. By the (proven) bounded height conjecture for semiabelian varieties (Theorem \ref{BHT}), we conclude that $h_L(r_{p,\sigma}) \leq c_3$ for some positive constant $c_3=c_3(S)$ that is independent of $p$. Using \eqref{equation::heightbound1}, we obtain that $2h_L(p) \leq c_1 c_3 + c_2$. The asserted finiteness follows hence from Proposition \ref{lemfinbh}.
\end{proof}

By these two above lemmas, it suffices to prove that there are at most finitely many points $p \in (\cC \cap G^{[2]})(\QQ)$ such that there exists some $\sigma \in \Gal(\overline{\KK}/\KK)$ satisfying $r_{p,\sigma} \in S(\QQ) \setminus (S^{\mathrm{oa}}(\QQ) \cup Z)$. By Lemma \ref{lemgalois2} (a), the fact that $r_{p,\sigma} \in C_i$ implies that $p_i+H_i$ is a torsion coset; this means that there exists $p_i^\prime\in \mathrm{Tors}(G_0)$ with $p_i+H_i = p_i^\prime+H_i$. (Note that we cannot assume that $p_i^\prime \in C_i(\QQ)$.)

{After relabeling, we may assume that $p_i+H_i$, $1\leq i \leq N$, is a torsion coset if and only if $1\leq i \leq N^\prime$.} For each $1\leq i \leq N^\prime$, we choose a torsion point $p_i^\prime \in \mathrm{Tors}(G_0)$ such that $p_i+H_i=p_i^\prime+H_i$.

\begin{lem}
\label{lemma::mainproof3}
There are at most finitely many points $p \in \cC(\QQ)$ contained in a torsion coset $q+H$ where $q \in \mathrm{Tors}(G)$ and $H \subset G$ is a subgroup of codimension $\geq2$ satisfying $H \not\supseteq H_i$ for all $i \in \{1, \dots, N^\prime\}$.
\end{lem}

\begin{proof}
Let $p \in \cC(\QQ)$ be contained in a torsion coset $q+H$ as in the lemma. Using Lemmas \ref{lemma::mainproof1} and \ref{lemma::mainproof2}, we may assume that there exists some $\sigma \in \Gal(\overline{\KK}/\KK)$ such that $r_{p,\sigma} \in S(\QQ) \setminus (S^{\mathrm{oa}}(\QQ) \cup Z)$. As observed above, we have $r_{p,\sigma} \in C_i(\QQ)$ for some $i \in \{1,\dots, N^\prime \}$.

By the minimality assumption on $H_i$, the translate $C_i^\prime = C_i - p_i^\prime \subseteq H_i$ cannot be $H_i$-anomalous in $C_i^\prime$. As $H \not\supseteq H_i$, the intersection $H \cap H_i$ is  a proper subgroup of $H_i$. Thus we infer $r_{p,\sigma} - p_i^\prime \in (C_i^\prime \cap H_i^{[1]})(\QQ)$ from the fact that $r_{p,\sigma}$ lies in $ (p_i^\prime + H_i) \cap (q^\sigma -q + H)$.

An application of Theorem \ref{BHT} for the curve $C_i^\prime \subseteq H_i$ yields an upper bound on the height $h_L(r_{p,\sigma} - p_i^\prime)$ that is independent of $p$. As $p_i'$ is torsion we have a bound on $h_L(r_{p,\sigma})$, and $r_{p,\sigma}\not\in Z$ implies that the height of $p$ is bounded. Proposition \ref{lemfinbh} yields the asserted finiteness.
\end{proof}

\subsection{Applying the inductive hypothesis}

We conclude the proof of Theorem \ref{mainthm} by dealing with the remaining points in $\cC \cap G^{[2]}$ by induction on $\dim(G)$. Note that this is the only part of the argument where the inductive hypothesis is actually used.

\begin{lem}
\label{lemma::mainproof4}
For each $i \in \{1, \dots, N^\prime\}$, there are at most finitely many points $p \in C(\QQ)$ contained in a torsion coset $q+H$ where $q \in \mathrm{Tors}(G)$ and $H \subset G$ is a subgroup of codimension $\geq2$ satisfying $H \supseteq H_i$.
\end{lem}

\begin{proof}
	Again, let $p \in \cC(\QQ)$ be contained in a torsion coset $q+H$ as in the lemma. We consider the quotient map $\varphi_i: G \rightarrow G/H_i$ and note that $\varphi_i(\cC)$ is not contained in a proper algebraic subgroup of $G/H_i$ because otherwise $C$ would be contained in a proper algebraic subgroup of $G$. Moreover, it is not a point. Indeed, this would imply that $\cC$ is contained in a coset of $H_i$ which would be strictly contained in a coset of $G_0$, contradicting the minimality of $G_0$.
	
	Now, as $p$ is sent to $\varphi_i(\cC) \cap (G/H_i)^{[2]}$ and $\dim(G/H_i)<\dim(G)$, the lemma follows from our inductive hypothesis.
\end{proof}

On combining Lemmas \ref{lemma::mainproof3} and \ref{lemma::mainproof4}, we obtain Theorem \ref{mainthm} immediately.

\subsection*{Acknowledgements:} The authors thank the referee for carefully reading the paper and providing several suggestions that significantly improved the article. They moreover thank thank Éric Gaudron and Philipp Habegger for comments and feedback. FB was supported by  the Swiss National Science Foundation grant 165525. LK was supported by an Ambizione Grant of the Swiss National Science Foundation. LK also received funding from the European Union Horizon 2020 research and innovation programme under the Marie Sklodowska-Curie grant agreement No. 101027237.

\bibliography{semiabelian}
\bibliographystyle{amsalpha}

\end{document}